\documentclass[11pt]{article}
\usepackage{a4wide,amsmath,amsbsy,amsfonts,amssymb,
stmaryrd,amsthm,mathrsfs,graphicx, amscd, tikz-cd, cancel, supertabular}
\usepackage[normalem]{ulem}
\usetikzlibrary{matrix,arrows,decorations.pathmorphing}

\usepackage{tikz}
\usetikzlibrary{matrix,arrows,decorations.pathmorphing}
\usepackage[all,cmtip]{xy}
\usepackage{qtree}

\usepackage[top=1.2in,bottom=1.2in,left=1.21in,right=1.21in]{geometry}
\usepackage{hyperref}

\newtheorem{theorem}{Theorem}[section]
\newtheorem{proposition}[theorem]{Proposition}
\newtheorem{corollary}[theorem]{Corollary}
\newtheorem{lemma}[theorem]{Lemma}

\theoremstyle{definition}

\newtheorem{definition}[theorem]{Definition}
\newtheorem{example}[theorem]{Example}

\newtheorem{question}[theorem]{Question}

\numberwithin{equation}{section}

\theoremstyle{definition}
\newtheorem{remark}[theorem]{Remark}
\newtheorem{remarks}[theorem]{Remarks}

\newcommand{\Ac}{\mathcal{A}}
\newcommand{\Cc}{\mathcal{C}}

\newcommand{\Ic}{\mathcal{I}}
\newcommand{\Rc}{\mathcal{R}}

\newcommand{\Abf}{\mathbf{A}}

\newcommand{\Ebf}{\mathbf{E}}
\newcommand{\Ibf}{\mathbf{I}}

\newcommand{\Kbf}{\mathbf{K}}
\newcommand{\Lbf}{\mathbf{L}}
\newcommand{\LL}{\mathbf{\Lambda}}
\newcommand{\Nbf}{\mathbf{N}}

\newcommand{\Ubf}{\mathbf{U}}
\newcommand{\Zbf}{\mathbf{Z}}
\newcommand{\Ps}{\mathscr{P}}

\newcommand{\Sc}{\mathcal{S}}

\newcommand{\sing}{\mathit{sg}}

\newcommand{\Cb}{\mathbb{C}}

\newcommand{\Fb}{\mathbb{F}}

\newcommand{\Pb}{\mathbb{P}}
\newcommand{\PP}{\mathbb{P}}

\newcommand{\Qb}{\mathbb{Q}}

\newcommand{\Rb}{\mathbb{R}}

\newcommand{\Zb}{\mathbb{Z}}
\newcommand{\Z}{\mathbb{Z}}

\newcommand{\Gr}{Gr}

\newcommand{\beq}{\begin{eqnarray}}
\newcommand{\eeq}{\end{eqnarray}}

\newcommand{\SL}{\textrm{SL}}
\newcommand{\tf}{\textrm{tf}}
\newcommand\ssm{\smallsetminus}

\newcommand{\para}[1]{\medskip\noindent\textbf{#1.}}

\DeclareMathOperator{\Dehn}{Dehn}

\DeclareMathOperator{\Aut}{Aut}
\DeclareMathOperator{\Def}{def}
\DeclareMathOperator{\ext}{Ext}

\DeclareMathOperator{\GL}{GL}

\DeclareMathOperator{\Hom}{Hom}

\DeclareMathOperator{\Orth}{O}

\DeclareMathOperator{\SO}{SO}

\DeclareMathOperator{\Teich}{Teich}

\DeclareMathOperator{\la}{\langle}
\DeclareMathOperator{\ra}{\rangle}

\DeclareMathOperator{\Diff}{Diff}
\DeclareMathOperator{\Mod}{Mod}

\DeclareMathOperator{\Tor}{{\mathcal I}}

\title{The Nielsen realization problem for K3 surfaces}
\author{Benson Farb and Eduard Looijenga \thanks{The first author was supported in part by National Science Foundation Grant No. DMS-181772 and the Eckhardt Faculty Fund. The second author is supported by the Chinese National Science Foundation. Both authors are supported by the Jump Trading Mathlab Research Fund. }}

\begin{document}
\maketitle

\begin{abstract}

The smooth (resp.\ metric and complex) Nielsen Realization Problem for K3 surfaces $M$ asks: when can a finite group $G$ of mapping classes of $M$ be realized by a finite group of diffeomorphisms (resp.\ isometries of a Ricci-flat metric, or automorphisms of a complex structure)?  We solve the metric and complex versions of Nielsen Realization,  and we solve the smooth version for involutions.
Unlike the case of $2$-manifolds, some $G$ are realizable and some are not, and the answer depends on the category of structure preserved. In particular, Dehn twists are not realizable by finite order diffeomorphisms.  We introduce a computable invariant $\Lbf_G$ that determines in many cases whether $G$ is realizable or not, and apply this invariant to construct an $S_4$ action by isometries of some Ricci-flat metric on $M$ that preserves no complex structure.

We also show that the subgroups of  $\Diff(M)$ of a given prime order $p$ which fix pointwise  some positive-definite $3$-plane in $H_2(M; \Rb)$ and preserve some complex structure on $M$ form a single conjugacy class in $\Diff(M)$ (it is known that then $p\in \{2,3,5,7\}$).
\end{abstract}

\tableofcontents

\section{Introduction}
Let $M$ be a closed, oriented, smooth manifold.  The (smooth) {\em mapping class group} of $M$ is the group $\Mod(M):=\pi_0(\Diff^+(M))$ of smooth isotopy classes of orientation-preserving diffeomorphisms of $M$.  The {\em Nielsen Realization Problem} asks if the natural map $\Diff^+(M)\to\Mod(M)$ has a section over any finite subgroup $G\subset \Mod(M)$; that is, does any finite group of mapping classes arise as a finite group of diffeomorphisms?  When there is such a section 
we say that $G$ is {\em realized} as a group of diffeomorphisms.   Whenever $M$ admits a particular kind of structure, there is a corresponding Nielsen realization problem. For example, one can also ask:

\bigskip
\noindent
{\bf Metric Nielsen Realization: } Suppose that $M$ admits a distinguished class $\Cc$ of metrics.  Given a 
finite subgroup $G\subset\Mod(M)$, does there exist a lift of $G$ to $\Diff(M)$, and a metric $g\in \Cc$, such that $G\subseteq {\rm Isom}(M,g)$?

\medskip
\noindent
{\bf Complex Nielsen Realization: }  Given a finite subgroup $G\subset\Mod(M)$, does there exist a lift of $G$ to $\Diff(M)$, and a complex structure $J$ on $M$, such that $G\subseteq\Aut(M,J)$?

\bigskip
For $M$ a closed, oriented surface $\Sigma_g$,  a positive answer in all categories was given by Fenchel for cyclic $G$ and by Kerckhoff in general: any finite subgroup 
$G\subset\Mod(\Sigma_g)$ is realized by a group of automorphisms of both a complex structure on $\Sigma_g$ and its associated constant curvature metric.  

In this paper we consider the Nielsen Realization Problem for K3 surfaces $M$.  These admit both complex structures and Ricci-flat metrics.  We solve the metric and complex 
Nielsen Realization (Theorem~\ref{thm:K3examples}) problems; we solve the smooth version for involutions (Theorem~\ref{theorem:main2}); and we give partial progress in general.    Unlike the case of $2$-manifolds, some $G$ are realizable and some are not, and the answer depends on the category of structure preserved. We introduce a computable invariant $\Lbf_G$ that determines in many cases if $G$ is realizable or not.  Two other results are: the classification of prime order subgroups of $\Diff(M)$ that preserve some complex structure and pointwise fix a positive definite $3$-plane in $H_2(M; \Rb)$; and a characterization of ``Nikulin type'' involutions in $\Diff(M)$.

\subsection{A computable invariant}\label{subsect:computable}

Recall that a {\em K3 surface} is a closed, simply-connected complex surface that admits a nowhere vanishing holomorphic $2$-form.  It is known that any two K3 surfaces are  diffeomorphic (see for example \cite{BPHV}, Cor.\ 8.6); we call this diffeomorphism type the {\em K3 manifold}.  We fix one such $M$ throughout this paper. It is also true that, as  Friedman-Morgan have shown (\cite{friedman-morgan}, Ch.\ VII Cor.\ 3.5), any complex structure on $M$ turns it into a K3 surface. The manifold $M$ has a natural fundamental class $[M]\in H_4(M;\Zb)$: it is the one for which the signature of $M$ is $-16$.  So every diffeomorphism of $M$ is necessarily orientation-preserving.  

It is known that $H_2(M):=H_2(M;\Zb)\cong\Zb^{22}$, and that $H_2(M)$ as a lattice equipped with the symmetric, bilinear intersection pairing (with respect to the above-mentioned orientation) 

\[
H_2(M)\times H_2(M)\to \Zb,  \quad  (a,b)\mapsto a\cdot b 
\]
is isomorphic to the {\em K3 lattice} $\Ubf^{\oplus 3}\oplus \Ebf_8(-1)^{\oplus 2}$, where $\Ubf$ denotes
the hyperbolic lattice (which has two isotropic generators $e,f$ with $e\cdot f=1$), the term $\Ebf_8(-1)$ denotes  the $\Ebf_8$-lattice whose pairing has been multiplied with $-1$, and $\oplus$ stands for the orthogonal direct sum. In particular $H_2(M)$ has signature type $(3,19)$.  We will often write $\LL$ for the lattice $H_2(M)$  and $\LL_{R}$ for $\LL\otimes R$ when $R$ is an abelian group. 
Note that the group $\Orth(\LL)$ of automorphisms of $\LL$ is an arithmetic lattice in the real semisimple Lie group $\Orth(\LL_{\Rb})$. 

The natural action of $\Diff(M)$ on $\LL$ factors through an orthogonal  representation 
$\rho:\Mod(M)\to\Orth(\LL)$. According to Donaldson (\cite{D-or} (3.20) and Definition (3.24)), the tautological $3$-plane bundle over the  space $\Gr^+_3(\LL_\Rb)$ of positive-definite $3$-planes in 
$\LL_{\Rb}$ comes with a natural orientation  (which we shall call a \emph{spinor orientation}), so that $\rho$ will take its values in the 
subgroup of $g\in  \Orth(\LL)$ that preserve this orientation. This subgroup is of index two in $\Orth(\LL)$ and will be denoted by $\Orth^+(\LL)$.   For example, if  $v\in \LL$ is such that  $v\cdot v=-2$ (we will refer to such a $v$ as a \emph{$(-2)$-vector} in $\LL$), then the orthogonal reflection with respect to $v$ is given by $x\in\LL\mapsto x+(v\cdot x)v$ and lies in $\Orth^+(\LL)$.  Such a reflection is the image under $\rho$  of a ``$2$-dimensional Dehn twist'' about a $2$-sphere in $M$ (see \cite{Se} for a precise definition of these).  Since these reflections  generate  all of $\Orth^+(\LL)$, the representation $\rho$ has image $\Orth^+(\LL)$  (see \cite{Bo} and \cite{Ma1}).

The theory of complex K3 surfaces produces many finite subgroups $G\subset\Mod(G)$ that are realizable by complex (even K\"ahler) automorphisms.  The question is whether or not these are the only realizable finite subgroups $G\subset\Mod(M)$.  We now introduce an invariant to help answer this question.

Let $G\subset \Orth^+(\LL)$ be a finite subgroup. Since $\Gr_3^+(\LL_{\Rb})$ is the symmetric space of $\Orth^+(\LL_{\Rb})$ it is nonpositively curved, and so $G$ will have a fixed point in this space. In other words, $G$ will leave invariant a positive definite $3$-plane $P\subset \LL_{\Rb}$. Since $G$ also preserves the orientation of $P$, the representation $\rho$ restricts to 
a representation  $\rho_{P}: G\to \SO(P)$. The image $\rho_P(G)$ will be conjugate to one of the familiar groups: cyclic (this includes the trivial group), dihedral, tetrahedral, octahedral or icosahedral. Note that only the last three cases yield
an irreducible representation.

 Let $\Ibf_G\subset\LL_{\Rb}$
be the sum of the  irreducible  $\Rb [G]$-submodules  of the type that appear in $P$. As our notation intends to suggest, this subspace (in contrast to $P$) is 
canonically associated with $G$.    

\begin{definition}[{\bf The invariant \boldmath$\Lbf_G$}]
 Define $\Lbf_G$ to be the following sublattice of $\LL$: 
\[\Lbf_G:=\Ibf_G^\perp\cap \LL.\] 
 
\end{definition}

The sublattice $\Lbf_G$ is clearly negative definite and $G$-invariant.  It is easily computable.  

\subsection{Metric and complex Nielsen Realization}
Siu \cite{Siu} proved that every complex K3 surface $M$ admits a K\"ahler metric  and  Yau \cite{Yau} proved that for every K\"ahler metric on $M$ there is a unique Ricci-flat K\"ahler metric in its cohomology class.   These  Ricci-flat metrics play the same role for K3 surfaces that constant curvature metrics play for Riemann surfaces.  
If a finite group $G\subset\Diff(M)$ preserves some complex structure on $M$ then one can take a $G$-average of a  K\"ahler metric  and then apply Yau's theorem to find a $G$-invariant, Ricci-flat K\"ahler metric.  Thus there are basically three Nielsen Realization problems for K3 surfaces, each stronger than the previous:  realization by diffeomorphisms of $M$; realization by isometries of some Ricci-flat metric on $M$; and realization by complex automorphism of $M$.

For the complex and Ricci-flat categories there is an associated  Teichm\"{u}ller space, which is particularly nice in the Ricci-flat case. For our purposes the latter case is the most relevant.  Let $\Teich(M)$ denote the 
moduli space of smooth isotopy classes of unit volume, Ricci-flat metrics on $M$; thus two such metrics define the same point of $\Teich(M)$ if one is the pullback of the other via a diffeomorphism of $M$ isotopic to the identity.  The space $\Teich(M)$ comes equipped with a natural $\Mod(M)$-action.  It also comes with a {\em period map}:  for a (Ricci-flat) Riemann
metric on $M$,  its space of self-dual harmonic  $2$-forms defines a positive definite $3$-plane in $H^2(M; \Rb)$; 
its orthogonal complement is negative definite and is defined by the space anti-self-dual harmonic $2$-forms.  Identifying $H^2(M; \Rb)$ with $\LL_\Rb$ via Poincar\'e duality, this gives a map 
\[
\Ps:  \Teich(M)\to \Gr^+_3(\LL_\Rb).
\]
Since the codomain of $\Ps$ is the symmetric space for $\Orth(\LL_\Rb)$, the group $\Orth^+(\LL)$ acts properly discontinuously on it.

We will need some known properties  of $\Ps$ which we recall in more detail in Proposition \ref{prop:Lo1}.  The image of $\Ps$ is the subset $U\subset \Gr^+_3(\LL_\Rb)$ given by the set of all positive-definite $3$-planes not perpendicular to a $(-2)$-vector. This $U$ is a simply-connected open subset of $\Gr^+_3(\LL_\Rb)$. The natural action of $\Diff(M)$ on $\LL$ factors through a representation  $\rho:\Mod(M)\to\Orth^+(\LL)$  whose kernel  $\Tor(M):=\ker(\rho)$ is called the {\em Torelli subgroup} of $\Mod(M)$. The map $\Ps$ is obviously $\Mod(M)$-equivariant, with $\Mod(M)$ acting on $ \Gr^+_3(\LL_\Rb)$ via $\rho$.  It maps every connected component of $\Teich(M)$ homeomorphically onto $U$,
thus making $\Teich(M)$ a smooth manifold of dimension $3\cdot 19=57$. If $\Teich(M)_o$  is a connected component of $\Teich(M)$, then 
the $\Mod(M)$-stabilizer $\Mod(M)_o$ of  $\Teich(M)_o$ maps isomorphically onto $\Orth^+(\LL)$. Thus $\Mod(M)$ acquires the structure of a semi-direct product 
\[
\Mod(M)=\Tor(M)\rtimes \Mod(M)_o\cong \Tor(M)\rtimes \Orth^+(\LL)
\]
and gives a corresponding description of $\Teich(M)$ as a $\Mod(M)$-manifold:
\[
\Teich(M)\cong \Mod(M)\times_{\Mod(M)_o} U.
\]
In particular, $\Tor(M)$ permutes simply-transitively the connected components of $\Teich(M)$.

A finite subgroup of $\Diff(M)$ that fixes some Ricci-flat metric on $M$ always fixes a point of $\Teich(M)$ and hence a connected component of 
$\Teich(M)$. In view of the preceding it is then natural to consider the metric Nielsen Realization problem for finite subgroups of  $\Orth^+(\LL)$,  although  this is clearly \emph{a priori} not the same as the corresponding problem for finite subgroups of $\Mod(M)$.   The following theorem solves this version of the Ricci-flat and complex versions of the Nielsen Realization problem for K3 surfaces.  The answer to the homological version is in terms of the invariant $\Lbf_G$ only.

\begin{theorem}[\textbf{Metric and complex Nielsen Realization}] 
\label{thm:K3examples}
Let $M$ be a K3 manifold and let $G$ be a finite subgroup  of   $\Orth^+(\LL)$.
\begin{enumerate}
\item $G$ lifts to the group of isometries of a Ricci-flat metric on $M$ if and only if  $\Lbf_G$ contains no
$(-2)$-vectors. 
\item $G$ lifts to the group of automorphisms of some complex structure on $M$ endowed with a Ricci-flat K\"ahler metric (making it a 
\emph{K\"ahler-Einstein} K3 surface) if and only if $\Lbf_G$ contains no
$(-2)$-vectors and in addition the trivial representation appears in $\Lbf_G^\perp$.  In this case the complex structure can be chosen so that $M$ is projective and $G$ acts by algebraic automorphisms.
\end{enumerate}
If instead $G$ is given rather as a finite subgroup of $\Mod(M)$, then a necessary and sufficient condition for it to lift as in Case 1 (resp.\  Case 2) is that  $G$ satisfies the homological conditions stated in that case and  also preserves a connected component of the Teichm\"uller space $\Teich(M)$. 
\end{theorem}

\begin{remark}\label{rem:}
We do not know whether conversely, a finite subgroup $G\subset \Diff(M)$ necessarily preserves a connected component of $\Teich(M)$, or equivalently (by Theorem \ref{thm:K3examples} above) a Ricci-flat metric.  Perhaps this might be proved by choosing a $G$-invariant metric and then use the Ricci flow to deform it equivariantly into a $G$-invariant 
Ricci-flat metric. (As Shing-Tung Yau pointed out to us, it suffices that one ends up with a metric of nonnegative scalar curvature, since for a K3 manifold this implies Ricci-flatness.)
\end{remark}

The following result shows that the metric and complex versions of Nielsen Realization are different. 

\begin{theorem}[{\bf Ricci-flat but not complex}]
\label{theorem:example1}
There exists a subgroup $G\subset\Orth^+(\LL)$ isomorphic to the alternating group $\Ac_4$ that can be realized as a group of isometries of some Ricci-flat metric on $M$, but that preserves no complex structure on $M$.
\end{theorem}

The proof of Theorem~\ref{theorem:example1} is an example of both a constructive and an obstructive application of Theorem~\ref{thm:K3examples}: the latter result implies that, to find the required example, it suffices to find a subgroup 
$G\subset\Orth^+(\LL)$ such that $\Lbf_G$ contains no $(-2)$-vector and $\Lbf_G^\perp$ contains no $G$-fixed vector.  This is a linear algebra problem.  

\para{Prime order subgroups} When $G$ pointwise fixes a positive-definite $3$-plane, it
will also fix a vector in $\LL$ with positive self-intersection, and then both parts of 
Theorem \ref{thm:K3examples} can be refined: an invariant Ricci-flat metric is always accompanied by an invariant complex structure for which this is a K\"ahler metric.  We also note that Theorem~\ref{thm:K3examples} implies 
the following.

\begin{corollary}\label{cor:classification}
A cyclic subgroup $G\subset\Diff(M)$ of prime order that leaves invariant a connected component of $\Teich(M)$  and  a 
positive-definite $3$-plane $P\subset \LL_\Rb$ but acts \emph{nontrivially} on $P$,  preserves a Ricci-flat metric. 
\end{corollary}

Indeed, in this case $\Lbf_G=\{0\}$ and so the assumption of part (1) is trivially satisfied.  
For this reason we will mostly focus on the case when $G$ fixes a positive-definite $3$-plane in $\LL_\Rb$  {\em pointwise}.(\footnote{In the literature, this condition is often called {\em symplectic}.}) Here part (2) of Theorem~\ref{thm:K3examples} can be both refined and be made more concrete. We first state a general result, and then indicate how these are realized on elliptically fibered K3 surfaces.  We should add however, that this is essentially hidden in the K3 literature and so here our main contribution is  to bring this to light and to present this in a uniform manner.

\begin{theorem}[{\bf Single conjugacy class}]
\label{theorem:classification}
Let $M$ be a K3 manifold. Let $G\subset\Diff(M)$ have prime order $p$ and fix pointwise a 
positive-definite $3$-plane in $\LL_\Rb$ and a complex structure on $M$ (or equivalently, a  Ricci-flat metric). Then $p\in\{2,3,5,7\}$ and for each such $p$, the subgroups of  $\Diff(M)$ as above  make up a single conjugacy class in $\Diff(M)$. Each such $G$ can be realized  geometrically inside the Mordell-Weil group of an elliptic fibration (see below).
\end{theorem}

To elaborate on the last point (and as will be described in detail in Subsection~\ref{sect:ellfib}), for each $(p,\nu)\in \{(2,8), (3,6), (5,4), (7,3)\}$ there exists a complex K3 surface $X$, elliptically fibered with 
$\nu$ fibers of Kodaira type  $I_p$ and $\nu$ fibers of  type $I_1$ (\footnote{Recall that an $I_m$-fiber is for $m>1$ a `cycle' of $(-2)$-curves, $m$ in number, and an $I_1$ fiber is a rational curve with a node.}), admitting an automorphism of order $p$; see Figure~\ref{figure:fiberaction}.  This automorphism is then realized inside its 
{\em Mordell-Weil group} (i.e.\ the group of automorphisms that preserve each fiber and act on it as a translation): it acts in each $I_p$-fiber  as a nontrivial `rotation' (but not necessarily over the same angle - see below). It also acts nontrivially in 
each $I_1$-fiber by ``rotation'' of the rational curve, fixing its unique singular point. See Figure~\ref{figure:fiberaction}.  This accounts for all the $G$-fixed points on $M$.  

The proof of Theorem~\ref{theorem:classification}, which as we noted above is implicit in the literature,  involves a substantial amount of nontrivial lattice theory. We will treat this aspect in a way that is both more conceptual and less computational than in the sources of which we are aware (so that this part might also be of some use for readers mostly interested in K3 surfaces). Along the way we find a construction of the Coxeter-Todd lattice that 
is possibly new.

\begin{figure}[h]
\includegraphics[scale=0.42]{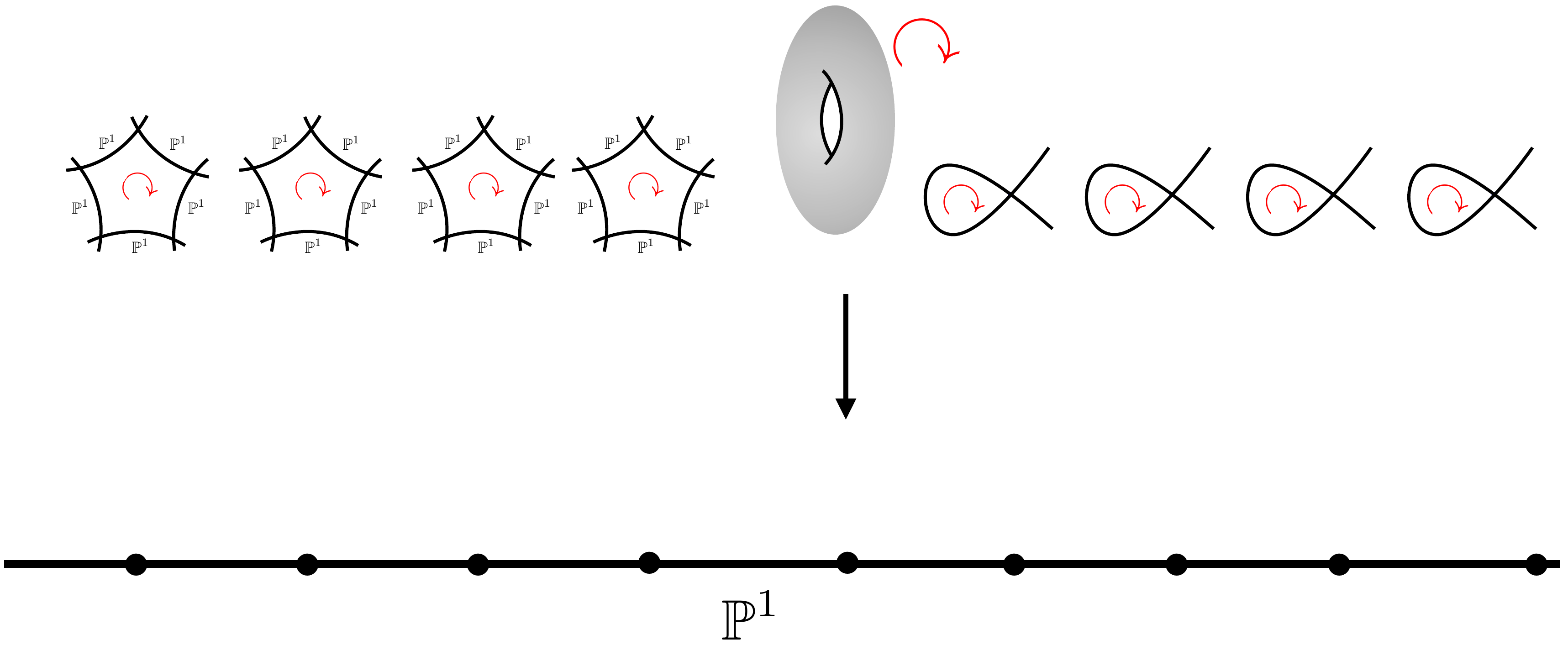}
\caption{\footnotesize The elliptically fibered K3 surface $X_5$.  The generic fiber is a smooth elliptic curve.  
There are $4$ fibers of type $I_5$, consisting of length $5$ $\Pb^1$-chains, and $4$ fibers that are rational nodal curves.  
The cyclic group of order $5$ acts on $X_5$ as the {\em Mordell-Weil group} of $X_5$; that is, the group of complex automorphisms of $X_5$ that preserve each fiber, translate each smooth fiber, rotate the $I_5$ fibers, and rotate the rational nodal curves fixing the singularity. The rotation angles on two of the $I_5$ (resp.\ $I_1$) fibers are twice that of the other two (cf.\ Remark~\ref{rem:ellfibrationcont}).}\label{figure:fiberaction}
\end{figure}

\begin{remark}
The complex structures on $M$ given up to a diffeomorphism that acts trivially on $H_2(M)$ are parametrized by  a connected, complex (but non-Hausdorff) manifold of complex dimension $20$.  Those that admit a group of complex 
automorphisms of order $p$ make up a locus of complex dimension $2\nu-4$  ($=12,8,4,2$ when $p=2,3,5,7$).  Theorem~\ref{theorem:classification}  then amounts to the statement that the irreducible components of this locus are transitively permuted by $\Orth^+(H_2(M))$.  Those that have an elliptic fibration make up a locus of  complex codimension $1$ (so of dimension $2\nu-5$), and those that have in addition a section a locus of complex codimension $2$ 
(and so of dimension $2\nu-6$; note this is then a finite set when $p=7$).  

In other words, there are families of prime order $G$ acting by complex automorphisms on K3s admitting no elliptic fibration (in fact this is generic among K3s with a $G$-action), and yet Theorem~\ref{theorem:classification} gives for each $p$ that any such action is smoothly conjugate to the specific elliptically fibered ``translation'' action given above (cf.\ Figure~\ref{figure:fiberaction}).
\end{remark}

\subsection{(Non)realization for cyclic subgroups}

Even in the smooth category, and even for involutions, there are obstructions for $G\subset\Mod(M)$ to be realized. 

\begin{theorem}[{\bf Smooth Nielsen (non)realization for involutions}]\label{theorem:main2}
Let $G\subset\Mod(M)$ be cyclic of order $2$. Then the nontrivial element of $G$ is  realized by a smooth involution  of $M$ if and only if  $\Lbf_G$ contains no $(-2)$-vectors. In this case the involution preserves some Ricci-flat K\"ahler structure on $M$.
\end{theorem}

\begin{remark}\label{remark:implications1}
Note that Theorem~\ref{thm:K3examples} gives the `if' part of the last sentence of Theorem~\ref{theorem:main2}. The main content of Theorem~\ref{theorem:main2} is therefore the `only if' part.
\end{remark}

An interesting example to which Theorem~\ref{theorem:main2} applies is a {\em Dehn twist}  $T_S:M\to M$ about an embedded $2$-sphere $S\subset M$ with $S\cdot S=-2$
(this intersection number is independent of the orientation of $S$).    It is known (see, e.g.\cite{Se}, Proposition 1.1)  that $T_S^2$ is smoothly isotopic to the identity. However, Theorem~\ref{theorem:main2} (or its proof, for the stronger result claimed) implies the following (which we will prove in detail below.)  

\begin{corollary}[{\bf Twists not realizable}]\label{thm:dehntwistproperty}
Let $T_S\in\Diff(M)$ be the Dehn twist about an embedded $2$-sphere $S\subset M$ with $S\cdot S=-2$.  Although $T_S^2$ is smoothly isotopic to the identity, $T_S$ is not even topologically isotopic to any finite order diffeomorphism of $M$.
\end{corollary}

One can also show for example that products of disjoint twists cannot be smoothly realized as finite subgroups of $\Diff(M)$.

\para{Involutions of Nikulin type}  Nikulin showed that the involutions of a $K3$-surface $X$ that act as the identity on $H^{2,0}(X)$ are all alike: they have eight fixed points, so that  the orbit space will have  eight-double points, and resolving these double points produces a $K3$-surface. The K3 surfaces endowed with such an involution make up a  single connected family.  As Morrison \cite{Mor} observed, the action of the  involution on their second homology is equivalent to the involution of the $K3$-lattice $\Ubf^{\oplus 3}\oplus \Ebf_8(-1)^{\oplus 2}$ that exchanges the last two summands.  This motivates the following definition.

\begin{definition}\label{def:nikulintype}
An element of order $2$ in $\Diff(M)$ is \emph{topologically} resp.\ \emph{differentiably of Nikulin type} if $f$ has eight fixed points points and the orbit space is, after a simple resolution 
of the resulting $8$ quotient singularities, a closed $4$-manifold homeomorphic resp.\  diffeomorphic to a $K3$-surface. 
\end{definition}

The main ingredient in our proof of Theorem~\ref{theorem:main2} is the following theorem, which characterizes involutions of Nikulin type in terms of purely homological data.  We remind the reader that any finite order 
diffeomorphism leaves invariant a positive $3$-plane in $\LL_\Rb$.

\begin{theorem}[{\bf Involutions of Nikulin type}]
\label{theorem:nikulin:type}
A smooth involution of the $K3$ manifold $M$ that fixes pointwise some positive $3$-plane in $\LL_\Rb$ is  topologically of Nikulin type and acts on $\LL$ as a Nikulin involution.
\end{theorem}

\para{Cyclic subgroups of order $p\geq 3$}
Unlike for the $p=2$ case (Theorem \ref{theorem:main2}) we do not have a definitive result for $p$ odd, but 
we are able to prove in many cases a dichotomy for realizable cyclic $G$ of prime order, into ``Nikulin type'' and ``Coxeter type''; see Theorem~\ref{thm:dich}.

\subsection{Concluding remarks}

We state here a basic result regarding the finite subgroups of $\Orth^+(\LL)$, which although not used  in this paper,  might be helpful for what follows. Let $G$ be a finite subgroup of $\Orth^+(\LL)$. The reflections  with respect to $(-2)$-vectors in the negative definite lattice $\Lbf_G$ generate  a finite Coxeter subgroup $W\subset \Orth^+(\LL)$. If  $C$ is a connected component of the complement of the union of the reflection hyperplanes in $\Lbf_G\otimes\Rb$ with respect to the $(-2)$-reflections in $W$ (a \emph{chamber}), then $G=G_C\rtimes (W\cap G)$, where $G_C$ denotes the $G$-stabilizer of $C$. The Torelli theorem for K3 surfaces implies that there exists a complex structure $J$ on $M$ such that $G_C$ is realized as a group of automorphisms of $(M,J)$. Although Theorem \ref{thm:dehntwistproperty} shows that any $(-2)$-reflection in $W$ cannot be realized as a smooth involution, it is possible  for $W\cap G$ to be nontrivial without containing such elements and so we cannot exclude the possibility that $G$ lifts to $\Diff(M)$, even when $W\cap G\not=\{1\}$.

\para{Methods} The methods used in this paper are varied. To constrain fixed sets we use the theory of finite group actions on manifolds 
(e.g.\  the Hirzebruch $G$-signature theorem and Smith theory by way of the work of Edmonds), and an adjunction formula of Kronheimer-Mrowka 
coming out of Seiberg-Witten theory. We analyze $\Mod(M)$ via its action on an appropriate Teichm\"{u}ller space, using a differential-geometric version of the Torelli theorem for K3 surfaces.  Finally, much of the work in this paper involves lattice theory.

\para{Relation to other work}  
Giansiracusa-Kupers-Tshishiku \cite{GKT} proved that the surjection $\pi:\Diff(M)\to \Mod (M)$ does not split over any finite index subgroup $\Gamma\subseteq\Mod(M)$.  Corollary \ref{thm:dehntwistproperty} implies 
this theorem when $\Gamma=\Mod(M)$; indeed it gives that $\pi$ does not even 
split over certain order $2$ subgroups of $\Mod(M)$.  We will see below, however, that 
this corollary can also be derived from a theorem of Bryan \cite{Bry}. We also mention that Baraglia and Konno \cite{BK} recently 
constructed a topological involution of $M$ which is not conjugate to a differentiable involution.

\bigskip
Since all finite group actions on a K3 manifold that we are aware of  fix a Ricci-flat metric we are led to ask:

\begin{question}\label{quest:}
Does every finite subgroup of $\Diff(M)$ leave invariant a Ricci-flat metric?
\end{question}

Another ``question'' (which we do not wish to advertise as explicitly as the one above, because it is too vague)  is the following.  As we indicated, the finite subgroups of $\Diff(M)$ that fix a positive-definite $3$-plane in $\LL_\Rb$ pointwise correspond to what in the 
K3 community are called symplectic automorphism groups. 
A well-known theorem of Mukai  \cite{Mu} characterizes the isomorphism classes of such groups as the subgroups  of the Mathieu group $M_{24}$ subject to certain properties ($M_{24}$ is a subgroup of the symmetric group $\Sc_{24}$ and the subgroups  $G\subset M_{24}$ in question 
are those that fix an element and have at least $5$ orbits).  Since a clear conceptual link between these two classes of finite groups  has not 
been found yet, one may wonder whether the passage to $\Mod(M)$ or $\Diff(M)$ can shed some light on this.

\begin{question}
The subgroup of $\Mod(M)$ generated by the Dehn twists defined by smoothly embedded $(-2)$-spheres in $M$ generate a normal subgroup 
$\Dehn(M)\subset \Mod(M)$. It is not hard to see that 
$\Dehn(M)$ contains the stabilizer of every connected component of $\Teich(M)$. This leads us to ask: what is $\Tor(M)\cap \Dehn(M)$?
Note that if that intersection is trivial, then we have a direct product decomposition $\Mod(M)=\Tor(M)\times \Dehn(M)$ with $\rho$ mapping
$\Dehn(M)$ isomorphically onto $\Orth^+(\LL)$. But it could well be that $\Dehn(M)$ is of finite index in $\Mod(M)$, if not all of $\Mod(M)$ just as for closed, connected, oriented 2-manifolds.
\end{question}

\para{Acknowledgements} We would like to thank Andras Stipsicz and Shing-Tung Yau 
for helpful email exchanges, and David Baraglia, Trevor Hyde and Jim Davis for comments and corrections on an earlier draft of this paper.  We are grateful to Curt McMullen for his extensive comments on that draft, and for his many useful suggestions.  Finally, we thank two anonymous referees for their comments and corrections.

\section{Metric and complex Nielsen Realization}\label{sect:nielsen}

The goal of this section is to prove Theorem~\ref{thm:K3examples} and Theorem~\ref{theorem:example1}.  

\para{Notation}
By a \emph{lattice} we shall always mean a free abelian group $\Lbf$ of finite rank endowed with symmetric bilinear form 
$(v,v')\in \Lbf\times \Lbf\mapsto v\cdot v' \in \Zb$.  In this context a direct sum is always understood as an orthogonal direct sum.

Unless mentioned otherwise, we will assume that a lattice $\Lbf$ is nondegenerate, so that the associated map from $\Lbf$ to its 
dual $\Lbf^\vee:=\Hom (\Lbf, \Zb)$ is injective. We then identify $\Lbf^\vee$ with the set of $x\in\Qb\otimes \Lbf$ for which $x\cdot v\in\Zb$ for all $v\in \Lbf$. 
The \emph{discriminant group} $\Lbf^\vee/\Lbf$ is finite and inherits from $\Lbf$ a symmetric  bilinear form  that takes values in $\Qb/\Zb$. 
In case the form is even, so that it comes from a quadratic form (namely $v\in \Lbf\to \frac{1}{2} v\cdot v \in \Zb$) and the discriminant group inherits such a form 
as well (again taking values in $\Qb/\Zb)$. If $n$ is a nonzero integer, then we denote by $\Lbf(n)$ the lattice whose underlying group is $\Lbf$, but
 for which the form has been multiplied by $n$. So $\Lbf(n)^\vee=\frac{1}{n}\Lbf^\vee$.

We make the notational distinction between the type of a root system of equal root length
and the lattice it generates (for which all roots have square length $2$) by in the latter case using a bold font. For example, $\Ebf_8$ is an even unimodular positive definite lattice of rank $8$ and
its vectors $v$ with $v\cdot v=2$ make up a root system of type $E_8$.

\subsection{The Torelli Theorem for K3 surfaces}
Before we begin the proof of Theorem~\ref{thm:K3examples} proper, we review the various formulations of the Torelli theorem for K3 surfaces (see for example \cite{Lo}).
Suppose $M$ is endowed with a complex structure $J$ (so  that   by  \cite{friedman-morgan} it becomes a K3 surface). It is known that  the automorphism group $\Aut (M, J)$ acts faithfully on $\LL$ and so can be considered as a subgroup of $\Orth^+(\LL)$. The complex structure gives rise to
a Hodge decomposition of $\LL$ with $\dim H^{2,0}(M,J)=1$, and basic Hodge theory tells us that this complex line is isotropic
for the complex-bilinear extension of the intersection pairing and positive definite for the hermitian extension. This implies
that  $H^{2,0}(M,J)$ maps under the projection $\LL_{\Cb}\to\LL_{\Rb}$ (taking the real part)  onto a positive-definite $2$-plane $\Pi_J\subset \LL_{\Rb}$.  This $2$-plane $\Pi_J$ inherits a natural orientation from $H^{2,0}(M,J)$, and as such it  is easily seen to be a complete invariant of the Hodge structure. 

The Torelli theorem states that more is true:
$\Pi_J$ determines $J$ up to a diffeomorphism that fixes $\Pi_J$ pointwise.  It also tells us that every oriented, positive-definite $2$-plane $\Pi$ so arises.
In other words, such a  $\Pi$ determines a complex structure $J$ up to diffeomorphism that fixes $\Pi$ pointwise. However that diffeomorphism may act nontrivially on $\LL$. Indeed, $\Aut (M, J)$ can be a proper subgroup of the $\Orth^+(\LL)$-stabilizer $\Orth^+(\LL)_\Pi$ of $\Pi$.  

In order to state a  more precise result, first note that
the orthogonal complement $\Pi^\perp\subset \LL_{\Rb}$ has hyperbolic signature $(1,19)$, so the set of $x\in \Pi^\perp$ with $x\cdot x>0$ consists of a pair of antipodal cones.  The spinor orientation and the given orientation of $\Pi$ single out one of them, denoted $C_\Pi$,  to which we will refer as the \emph{positive cone}.
The  subgroup $W_{\Pi}\subset \Orth^+(\LL)_\Pi$ generated by the orthogonal reflections in the  $(-2)$-vectors in $\LL\cap \Pi^\perp$ preserves  $C_\Pi$ and acts on it as a Coxeter group. The collection of  reflection hyperplanes is locally finite on $C_\Pi$ and decomposes $C_\Pi$ into faces.
The group $W_\Pi$ acts simply transitively of the set of open faces. We refer to these as \emph{$\Pi$-chambers}. 

One way to phrase a strong form of the Torelli theorem is the assertion that for every $\Pi$-chamber $K\subset C_\Pi$ there exists
a complex structure $J$ for which $\Pi_J=\Pi$ and $K$ is the set of classes of K\"ahler forms relative to $J$.  This complex structure is indeed unique up to a diffeomorphism that acts trivially on $\LL$. Moreover, $\Aut(M,J)$ maps isomorphically onto the $\Orth^+(\LL)$-stabilizer of the pair $(\Pi,K)$.
In particular, to give a finite subgroup of  $\Aut(M,J)$ is to give a finite subgroup $G$ of  $\Orth^+(\LL)_{\Pi,K}$. Since $K$ is a convex (open) cone, this is equivalent to asking that $G$ has  fixed point $\kappa\in K$, which is then the class of a K\"ahler form.

There is a further refinement of the above that involves the twistor construction and that is relevant here. According to Yau's solution of the Calabi conjecture, $\kappa$ 
 is representable by a unique K\"ahler form for which the associated metric $s$ is Ricci-flat
(so that we have  a Ricci-flat K\"ahler metric). As a consequence, $G$ will fix this Ricci-flat K\"ahler metric. The harmonic $2$-forms that are self-dual for the $\star$-operator  define the $3$-dimensional positive-definite space 
$P=\Pi +\Rb \kappa$.  Note that by  construction, $P$ is not perpendicular to any $(-2)$-vector. The twistor construction (see for example \cite{HKLR})  shows that the $\SO (P)$-orbit of the pair $(\Pi, \kappa)$ defines a family of Ricci-flat K\"ahler structures on $M$ for which $s$ is the associated metric.  This enables us to state the  Torelli theorem in differential-geometric terms: For the purposes of reference we formulate this as a proposition.

\begin{proposition}[{\bf The $3$-plane of self-dual harmonic forms}]\label{prop:Lo1} 
Given any Ricci-flat metric $s$ on $M$ for which $M$ has unit volume, let $P_s$ be the positive-definite $3$-plane of harmonic (in the $s$ metric) self-dual $2$-forms.  
Then :
\begin{enumerate}
\item $P_s$ is not perpendicular to any $(-2)$-vector.
\item $P_s$ determines  $s$ on $M$ up to an element of the Torelli group. 
\item Any positive-definite $3$-plane $P\subset \LL_{\Rb}$ that is not perpendicular to a $(-2)$-vector so 
arises.
\end{enumerate}
\end{proposition}

For a proof of Proposition~\ref{prop:Lo1}, see for example \cite{Lo}. The following is a corollary of the twistor construction.

\begin{corollary}\label{cor:isometrylift}
Let $s$ be a Ricci-flat metric on $M$ and let $P_s\subset \LL_{ \Rb}$ be the positive definite $3$-plane defined by its harmonic self-dual $2$-forms.
Then the isometry group of $(M,s)$ is finite and  its representation on $\LL$ identifies this  group with the $\Orth^+(\LL)$-stabilizer of $P_s$.
\end{corollary}
\begin{proof}

Suppose $g\in \Orth^+(\LL)$ leaves $P_s$ invariant. Then $g$ must have finite order. It acts on $P_s$ as an element of $\SO(P_s)$ and so  $P_s^g\not= \{0\}$. For each $\kappa\in P_s^g\ssm \{0\}$ the natural orientation on $P_s$ induces a natural orientation on $\Pi_\kappa:=\kappa ^\perp\cap P_s$.  By the twistor construction,  $\Pi_\kappa$ defines a unique complex structure $J_\kappa$ on $M$ for which  $\Pi_{J_\kappa}=\Pi_\kappa$ as oriented planes and  for which $s$ is a K\"ahler metric.  

The Torelli theorem then asserts 
that $g$ has a unique representative $g_\kappa\in \Diff (M)$ that preserves both $J_\kappa$ and $s$ (or equivalently, $\kappa$ and $s$). Since the isometry group of $(M,s)$ is discrete,
$g_\kappa$ is independent of the connected component of $\kappa$ in $P_s^g\ssm \{0\}$.  If $P_s^g\ssm \{0\}$ is not connected then it is a punctured line
and $\kappa\mapsto-\kappa$ exchanges its components. This has the effect of changing the orientation of $\Pi_\kappa$ and hence replaces 
the complex structure $J_\kappa$ by its conjugate $\bar J_\kappa$. This will not affect $g_\kappa$ (for $g_\kappa$ also preserves $\bar J_\kappa$) and so there is a unique lift $\tilde g$ of $g$ to an isometry of $(M,s)$. Hence $g \mapsto \tilde g$ defines a group homomorphism that is a section of the
map $\Aut(M,s)\to \Orth^+(\LL)_{[P_s]}$. The uniqueness assertion implies that the latter is injective.
\end{proof}

\subsection{Completing the proofs of metric and complex Nielsen realization}

With the above in hand we can complete the proof of Theorem~\ref{thm:K3examples}.

\begin{proof}[Proof of Theorem~\ref{thm:K3examples}]
We prove the nontrivial directions only (in the mapping class version we assume that $G$ fixes a connected component of $\Teich(M)$). It then remains to prove that if $\Lbf_G$ contains no $(-2)$-vectors then $G$ is realized as a group of isometries of a Ricci-flat metric. The other properties already follow from our discussion above.

In the Grassmannian  of $G$-invariant $3$-planes $P\subset \Lbf_G^\perp\otimes\Rb$, those  
with the property that the rational hull of  $P$  is $\Lbf_G^\perp\otimes\Rb$ (or equivalently, for which $P^\perp\cap \LL=\Lbf_G$) are dense. Hence we can find a $G$-invariant,  positive-definite $3$-plane 
$P\subset \Lbf^\perp_G\otimes\Rb$ for which $P^\perp\cap \LL=\Lbf_G$. In that case $P^\perp$ contains no $(-2)$-vectors. By the differential-geometric version of the Torelli theorem, there exists then a Ricci-flat 
metric $s$ on $M$ which  in the mapping class version  can be taken to be in the component of $\Teich(M)$ fixed by $G$ such that $P$ is the space  defined by the self-dual harmonic $2$-forms. Corollary  \ref{cor:isometrylift} implies that $G$ then lifts to a group of isometries of $(M,s)$.
\end{proof}

If $G$ acts irreducibly on $P$, and if $P^\perp$ does not contain any $(-2)$-vector, then $G$ preserves a Ricci-flat  metric without necessarily preserving a complex structure on  $M$. This explains why we put the emphasis on the 
differential-geometric version of the Torelli theorem.  Theorem~\ref{theorem:example1}, which we now prove, gives an example that shows that this possibility actually occurs.

\begin{proof}[Proof of Theorem~\ref{theorem:example1}]
We think of the symmetric group $\Sc_4$ as the Weyl group of a root system $R$ of type $A_3$. It preserves the root lattice $\Abf_3$, a free abelian group of rank $3$.   $\Sc_4$ acts on the dual $\Abf_3^\vee$ by means of the contragredient representation. 
Since the  reflection representation of $\Sc_4$ is self-dual over $\Qb$, these two representations are equivalent when tensored with $\Qb$.
The direct sum $\Abf_3\oplus \Abf_3^\vee$ comes with a natural  quadratic form defined by $f(x, \xi):=\xi (x)$. This turns $\Abf_3\oplus \Abf_3^\vee$
into a lattice isomorphic to $\Ubf^{\oplus 3}$ and thus produces a representation of the alternating group  $\Ac_4$ on $\Orth(\Ubf^{\oplus 3})$.  

After tensoring with $\Rb$, this representation becomes a direct sum of two copies of the reflection representation, and so
$\Sc_4$  preserves a positive-definite $3$-plane  $P\subset\Ubf_{\Rb}^{\oplus 3}$. The subgroup $\Ac_4\subset \Sc_4$ will act on $P$ in an orientation-preserving manner. In order to extend this action to the $K3$ lattice 
$\Ubf^{\oplus 3}\oplus \Ebf_8(-1)^{\oplus 2}$,  note that there exist two perpendicular copies of the $A_3$ root system in an $E_8$ root system 
such that the associated embedding $\Abf_3\oplus\Abf_3\hookrightarrow \Ebf_8$ has the property that its orthogonal complement is spanned by two mutually perpendicular $(+4)$-vectors.
So if we let $\Ac_4$ act on each summand of $\Ebf_8(-1)^{\oplus 2}$ through this embedding, then the resulting action of $\Ac_4$ on the 
$K3$-lattice $\Ubf^{\oplus 3}\oplus \Ebf_8(-1)^{\oplus 2}\cong \LL$ has the property that $\Lbf_G$ is spanned by $4$ mutually perpendicular $(-4)$-vectors.
Theorem \ref{thm:K3examples} implies that  $G$ is realized by a group of isometries of $M$ relative to some Ricci-flat metric, but cannot fix a complex structure.
\end{proof}

\section{The fixed set of a diffeomorphism of prime order}

The goal of this section is to prove some results on the structure of the fixed-point set of  a finite subgroup $\Diff(M)$ of prime order $p$.  In the first two subsections we introduce some basic tools we need in order to do this.

Throughout this section, $N$ will be a connected, closed, oriented,  smooth $4$-manifold and $G\subset \Diff^+(N)$ will stand for a finite cyclic group of prime order $p\geq 2$ of orientation-preserving diffeomorphisms.  
Since $G$ is a finite group acting by diffeomorphisms, the fixed set $N^G$ is an embedded, closed, smooth submanifold of $N$.

\subsection{The Hirzebruch $G$-signature theorem}\label{subsect:HZsign}

The first tool we will need is the Hirzebruch $G$-signature theorem, which we now describe.  The assumptions are as given in the previous paragraph.  

Since the $G$-action on $N$ preserves orientation, 
each path component of $N^G$ has dimension $0$ or $2$.  At an isolated fixed point $z\in N^G$, the $G$-action on the tangent space $T_zN$
preserves a complex structure that is compatible with the given orientation, and is then given as a $2$-dimensional complex $G$-representation
by two nontrivial complex characters $\chi_1, \chi_2$, so  equivalent to the
complex representation 
\begin{equation}
\label{eq:localrep}
g\in G: (z_1,z_2)\mapsto (\chi_1(g)z_1,\chi_2(g)z_2).
\end{equation}
This is not unique: there is exactly one other complex structure with this property and that is the conjugate of this one for which the associated characters 
are of course $(\bar\chi_1,\bar\chi_2)$. Therefore  the {\em signature defect} ${\rm def}_z$ of the $G$-action at $z$  defined by 
\begin{equation}
\label{eq:defect1}
\Def_z:=\sum_{g\in G\ssm \{1\}}\frac{(1+\chi_1(g))(1+\chi_2(g))}{(1-\chi_1(g))(1-\chi_2(g))}.
\end{equation}
is intrinsic. Since $G$ is cyclic of prime order $p$,  $\chi_1$ and $\chi_2$ are generators of its group of complex characters and so $\chi_2=\chi_1^q$ for some $q\in \Fb_p^\times$. We can then also write this as   
\[
\Def_z=\sum_{\zeta\in \mu_p\ssm\{1\}}\frac{(1+\zeta)(1+\zeta^q)}{(1-\zeta)(1-\zeta^q)}\tag{$3.2'$}
\]
Note that each term in this sum is real, and is the norm squared of $(1+\zeta)/(1-\zeta)$ when $q=-1$.  For future reference we observe that the Cauchy-Schwarz inequality implies that
\[
2\frac{(1+\zeta)(1+\zeta^q)}{(1-\zeta)(1-\zeta^q)}\le \frac{(1+\zeta)(1+\zeta^{-1})}{(1-\zeta)(1-\zeta^{-1})}+
\frac{(1+\zeta^q)(1+\zeta^{-q})}{(1-\zeta^q)(1-\zeta^{-q})},
\]
with equality holding only when $(1+\zeta^q)/(1-\zeta^q)=(1+\zeta^{-1})/(1-\zeta^{-1})$, which simply means that $q=-1$. It follows that $\Def_z$ takes its maximal value for $q=-1$ and for  $q=-1$ only.

For each $2$-dimensional path component $C\subseteq N^G$, the {\em signature defect} of the $G$-action at $Y$ is defined by (cf.\ \cite{HZ}, p.178, (15)): 
\begin{equation}
\label{eq:defect2}
\Def_C:=\frac{\displaystyle (p^2-1)}{\displaystyle 3} C\cdot C,
\end{equation}
where $C\cdot C$ denotes the intersection product in $H_2(N;\Zb)$: it is by definition zero unless $C$ is orientable, in which case we choose a generator $[C]$ of the infinite cyclic group $H_2(C;\Zb)$ and then take $[C]\cdot [C]$. So when $p=2$, we get $\Def_z=0$ and $\Def_C=C\cdot C$.

Since $N/G$ is an oriented rational homology manifold, it has a well-defined signature $\sigma(N/G)$, but for our purpose we can just as well define it as the signature of the restriction of the intersection pairing to $H_2(N; \Rb)^G$. 
 
\begin{theorem}[{\bf Hirzebruch Signature Formula (see, e.g. (12) on page 177 of \cite{HZ})}]
\label{theorem:Hirzebruch}
Let $N$ be a closed, oriented $4$-manifold.  Suppose that $G$ is a finite subgroup of the group $\Diff^+(N)$ of orientation-preserving diffeomorphisms of $N$ of prime order $p$.  Then
\begin{equation}
\label{eq:gsig1}
\textstyle p\cdot\sigma(N/G)=\sigma(N)+\sum_z{\rm def}_z +\sum_C{\rm def}_C, 
\end{equation}
where $\sigma$ denotes the signature, the first sum ranges over the $0$-dimensional path components $z$ of $N^G$, and the second sum ranges over the $2$-dimensional path components $C$ of $N^G$.
\end{theorem}

\subsection{$H_2(N;\Zb)$ as a $\Zb[G]$-module}
\label{subsection:h2}
When $N$ is simply-connected, a result of Edmonds gives more precise information 
about the $2$-dimensional components of $N^G$ in terms of the $G$-action on $H_2(N;\Zb)$. In order to state this we will need a few facts regarding $\Zb[G]$-modules. 

If $G$ is a finite cyclic group of prime order  $p$, let $\Ic_G\subset\Zb[G]$ denote the augmentation ideal and let $\check{\varepsilon}_G:=\sum_{g\in G} g\in\Zb[G]$ be the co-augmentation. Then $\Qb[G]=\Qb \Ic_G\oplus \Qb \check{\varepsilon}_G$ decomposes the algebra into a direct sum of the cyclotomic field
of degree $p-1$ (which we shall denote by $\Qb(G)$) and $\Qb$. 
Any finitely-generated $\Qb[G]$-module decomposes into a  trivial module 
and copies of the cyclotomic module $\Qb(G)$. 

There is an integral counterpart of this. First note that the image of $\Zb [G]$ in  $\Qb(G)=\Qb[G]/\Qb \check{\varepsilon}_G$ is 
 its ring of integers, denoted here by $\Zb (G)$. We shall refer to  $\Zb(G)$ as the \emph{basic integral cyclotomic module}. As recalled by Edmonds in \cite{Ed},  every finitely-generated  module $H$ over the group-ring 
$\Zb[G]$ which is free abelian over $\Zb$ decomposes into a trivial module and  \emph{indecomposable} summands that are either isomorphic to an ideal in $\Zb(G)$, 
called \emph{integral cyclotomic}  or of $\Zb$-rank $p$, called \emph{regular}. For the primes of interest here ($p=2,3,5,7$) the ideal class group of $\Qb(\mu_p)$ is trivial, so that then  every indecomposable module is principal; that is, is isomorphic to $\Zb$, 
$\Zb(G)$ or $\Zb[G]$, respectively. Beware however that a decomposition  into isogeny types need not be unique. 

\begin{example}\label{example:}
A  Coxeter transformation $c\in \Orth(\Ebf_8)$ is of order $30$ and has as its 
eigenvalues the eight primitive  $30$th roots of unity.
So for $p=2,3,5$, the transformation $c^{30/p}$ has prime order $p$ and does not have $1$ as an 
eigenvalue. This makes  $\Ebf_8$ a purely cyclotomic $\Zb(\mu_p)$-module; its rank as a  $\Zb(\mu_p)$-module is $8/(p-1)$, that is $8, 4$ and $2$, respectively.
\end{example}

The obvious ring homomorphism $\psi:\Zb[G]\to \Zb(G)\times \Zb$ satisfies 
$\psi(\check{\varepsilon}_G)=(0,p)$ and of course $\psi(1)=(1,1)$.   
Thus $\psi$ becomes a ring isomorphism if we invert $p$. In other words, a regular $\Zb[p^{-1}][G]$ module splits into a cyclotomic and a trivial module. Put differently, the distinction between a regular module and a sum of a cyclotomic module and a trivial module survives reduction modulo $p$. Indeed,  since a $p$th root of unity in an algebraic closure of $\Fb_p$ equals $1$,  tensoring  with $\Fb_p$ renders the $G$-action unipotent  and an indecomposable regular resp.\ cyclotomic module is then given by a single Jordan block  of size $p$ resp.\ $p-1$. Here is simple, but useful observation.

\begin{lemma}\label{lemma:nocyclotomic}
Let $H$ be a free abelian group of finite rank  and $G\subset \GL(H)$ a subgroup of prime order $p$. Assume that as a $\Zb[G]$-module, $H$ has only regular and trivial summands and that the ideal class group of $\Qb(\mu_p)$ is trivial. If  $g\in G$ is a generator, then the image of 
$g-1: H\to H$ is a direct summand of $H$ as a $\Zb$-module (and equal to the intersection $H^{\not=1}$ of  $H$ with the subspace of $H_\Cb$ where $g$ acts with eigenvalues $\not=1$).

If $H$ is endowed with a $G$-invariant unimodular form, then $H^{\not=1}=(H^G)^\perp$ and has discriminant lattice an $\Fb_p$-vector space of dimension equal to the number of regular summands in $H$.
\end{lemma}

\begin{proof}
Our assumptions imply that $H$ is as a $\Zb[G]$-module a direct sum of copies of  $\Zb[G]$ and trivial $\Zb[G]$-modules, and so it suffices to check this
for the case $H=\Zb[G]$.  But then the image of $g-1$ is the augmentation ideal $\Ibf_G\subset \Zb[G]$ and is indeed a direct summand.

Now assume that $H$ comes with $G$-invariant unimodular form and choose a $\Z[G]$-linear isomorphism $H=\Zb[G]^\nu \oplus \Zb^\mu$. 
The projector $\frac{1}{p}\check{\varepsilon}_G$ defines 
orthogonal projection onto $H^G_\Qb$. It is clear that the image of this projector is $\la \frac{1}{p}\check{\varepsilon}_G\ra^{\oplus\nu}\oplus  \Zb^\mu$. Since $H$ is unimodular, the image of this projection can be identified with $(H^G)^\vee$. It follows  that $(H^G)^\vee/H^G$ is a $\Fb_p$-vector space of dimension $\nu$. Since $H^{\not=1}$ is the orthogonal complement of $H^\perp$ in $H$ its discriminant group is canonically isomorphic to that of $H^\perp$.
\end{proof}

We can now state the following. 

\begin{proposition}[Edmonds, (Prop.\ (2.4) of \cite{Ed})]\label{prop:ed}
Let $N$ be a closed, oriented, simply-connected $4$-manifold and $G\subset \Diff^+(N)$ a finite cyclic subgroup of prime order $p$.
Let $t,c$ and $r$ denote the number of trivial, cyclic and regular summands of $H^2(N;\Zb)$ as a $\Zb[G]$-module.   Then 
\[b_0(N^G; \Fb_p)+b_2(N^G; \Fb_p)=t+2\]
and
\[b_1(N^G; \Fb_p)=c,\]
where  $b_i(N^G; \Fb_p)$ stands for the $i$th Betti number with $\Fb_p$-coefficients, $\dim_{\Fb_p} H^i(N^G; \Fb_p)$. 
\end{proposition}

 \subsection{Cyclotomic summands and topology of the fixed set}
In this subsection we concentrate on  the case when $G$ is a cyclic subgroup of order $p$ of $\Diff (M)$, with $p$ prime. We denote a generator of $G$ by $f$.  

Recall that a {\em spin structure} on a $4$-manifold $N$ is a lift of the structure group $\SO(4)$ of the tangent bundle $TN$ to its universal ($2$-sheeted)
cover $\widetilde{\SO}(4)$ (we will review this notion in more detail in \S~\ref{subsect:resolution}).

\begin{lemma}[{\bf Basic properties of \boldmath$M^G$}]
\label{lemma:basics1}
Let $G\subset \Diff(M)$ be a finite cyclic subgroup.  Then:
\begin{enumerate}
\item $G$ is orientation preserving. 
\item $M$ is a spin manifold and $G$ preserves a spin structure on $M$.
\item $M^G$ is a finite union of points and closed, orientable $2$-manifolds.
\end{enumerate}
\end{lemma}

\begin{proof}
The first property was already noted in the introduction and appears here for the sake of record only.

Since $M$ is a K3 manifold, $\pi_1(M)=0$ and $\omega_2(M)=0$.  It follows that $M$ has a unique spin structure, so that any diffeomorphism must preserve this spin structure, proving (2).

By averaging any smooth Riemannian metric under the $G$ action, $G$ preserves some Riemannian metric on $M$. It follows that $M^G$ is a 
(possibly disconnected) closed manifold.  Since $G$ is orientation-preserving, it acts on the normal bundle of any component $C$ of $M^G$ as  special orthogonal transformation without eigenvalue $1$. This implies that such a component has even dimension (it is a finite union of points and closed $2$-manifolds). This argument also shows that if $C$ is a surface component and $G$ has odd order, then this normal bundle must be orientable (so that $C$ is orientable). If $p=2$ then Proposition 3.2 of \cite{Ed} implies, since $G$ is orientation-preserving and spin-preserving, that $M^G$ is orientable.(\footnote{In this and in other results quoted from \cite{Ed}, the hypothesis is actually that the $G$-action is locally linear.  This is satisfied for any diffeomorphism.}) This proves (3).
\end{proof}

Denote by $\nu$ the multiplicity of the cyclotomic character of the $G$-action on $H_2(M;\Qb)$.  Note that the cyclotomic character has degree $p-1$, 
which implies that 
\[\nu (p-1)\le 22-3=19.\]

\begin{proposition}[{\bf Topology of $M^G$}]
\label{prop:fixtop1}
Let $G$ be cyclic subgroup of prime order $p$ of the  diffeomorphism group of the K3 manifold $M$ which 
pointwise fixes a positive $3$-plane in $H_2(M; \Rb)$. Denote by
$\nu$ the multiplicity of the rational cyclotomic character of the $G$-action on $H_2(M;\Qb)$.  Then the  following hold:
\begin{enumerate}
\item The Euler characteristic of $M^G$ is
\begin{equation}\label{eqn:eulerchar}
\chi(M^G)=24-\nu p.
\end{equation}
\item The signature of $M/G$ is  \[\sigma(M/G)=\sigma(H_2(M;\Rb)^G)=\nu(p-1)-16.\]
\item The signature defect of the $G$-action on $M$ is 
\begin{equation}\label{eqn:signdefect}
\textstyle \sum_z{\rm def}_z +\sum_C{\rm def}_C=(p-1)(\nu p-16).
\end{equation} 
\end{enumerate}
\end{proposition}

\begin{proof}
Since $H_2(M;\Rb)\cong\Rb^{22}$ decomposes into $\nu$ cyclotomic summands 
(each of which has dimension $p-1$) and $22-(p-1)\nu$ trivial summands, and since the character of each cyclotomic summand equals $-1$, it follows that the Lefschetz number of a generator $f$ of $G$ is 
\[L(f)=1+1\cdot (22-(p-1)\nu)+(-1)\cdot \nu+1=24-\nu p.\]
But  $\chi(M^G)=\chi(M^f)=L(f)$, the first equality following since $f$ generates $G$; the second following from the Lefschetz Fixed-Point Theorem.  This proves (1).

As noted in \S~\ref{subsection:h2},  $H_2(M;\Qb)$ decomposes as a direct sum of $1$-dimensional trivial and $(p-1)$-dimensional cyclotomic $G$-representations.  By definition there are $\nu$ and so $22-\nu(p-1)$ of the former.  Since we assumed that $H_2(M;\Rb)^G$ contains positive-definite $3$-plane, it follows that the intersection form on $H_2(M;\Rb)$ restricts to a bilinear form on $H_2(M;\Rb)^G$ of type 
$(3,19-\nu(p-1))$.  Thus $\sigma(H_2(M;\Rb)^G)=\nu(p-1)-16$, proving (2).   

Given this, the $G$-signature Formula (Theorem~\ref{theorem:Hirzebruch}) implies 
\[
\textstyle \sum_z{\rm def}_z +\sum_C{\rm def}_C=p (-16+\nu(p-1))+16=(p-1)(\nu p-16),
\]
proving (3).
\end{proof}

\subsection{Smooth involutions have finite fixed set}
In this subsection we focus on the case $p=2$.  \emph{Throughout this section we also assume that $G$ fixes pointwise a 
positive-definite $3$-plane in $\LL$.}  This implies that $\Lbf_G$ is the orthogonal complement of $H^2(M)^G$ in $H^2(M)$ and that $\Lbf_G$ is negative definite.

The first part of the following proposition is due to Bryan (\cite{Bry}, Section 4.2) and was obtained by him in a somewhat different manner. It has in common with our proof that it is based on Seiberg-Witten theory.

\begin{proposition}\label{proposition:involution}
Let $G\subset\Diff(M)$ be a subgroup of order $2$.  Then the fixed set $M^G$ consists of $8$  points and the  $\Zb[G]$-module $H_2(M)$ has $8$ regular summands, $6$ trivial summands and no cyclotomic summands.
\end{proposition}

\begin{proof}
Let  $P\subset H^2(M; \Rb)$ a positive-definite $3$-plane  pointwise fixed by $G$. By Lemma~\ref{lemma:basics1}, $M^{G}$ is a finite union of points and closed surfaces.  That lemma also tells us that the $G$-action preserves a spin structure on $M$.  A result of Atiyah-Bott (\cite{AB}, Proposition 8.46) states that then all components of $M^{G}$ have the same 
dimension: $0$ or $2$. We first exclude the last case. Suppose otherwise  and  assume that $M^{G}$ is a finite union $C_1 \cup \cdots \cup C_m$ of closed, connected surfaces, $m\geq 1$.

 Since all components of $M^{G}$ have the same dimension, Corollary 2.6 of Edmonds \cite{Ed} implies that the classes of any $m-1$ elements  of $\{C_1,\ldots ,C_m\}$ in  $H_2(M;\Fb_2)$ span an $(m-1)$-dimensional subspace.   In particular each $C_i$ defines a nonzero class in $H_2(M;\Fb_2)$.  Since $C_i$ is orientable by Lemma~\ref{lemma:basics1}, it will then even support a nonzero integral class.  

Formulas (1) and (3) of Proposition \ref{prop:fixtop1} give $\sum_i \chi (C_i)=24-2\nu$ and $\sum_i C_i\cdot C_i= 2\nu-16$, and  so they combine to give
\begin{equation}\label{eqn:noether}
\sum_i\big(\chi(C_i)+C_i\cdot C_i\big)=8
\end{equation} 
We claim that each term on the left-hand side of \eqref{eqn:noether} is $\le 0$; this evidently gives a contradiction.
If $C_i\cdot C_i<0$, then the fact that the intersection form on $M$ is even implies that $C_i\cdot C_i\le -2$,  and hence $\chi(C_i)+C_i\cdot C_i\le 0$. On the other hand, when $C_i\cdot C_i\ge 0$, then as we have seen,  $C_i$ supports a nonzero integral homology class, and then the adjunction inequality  of Kronheimer-Mrowka (Theorem 1.1 of \cite{KM}) asserts that it is then still true that
$\chi(C_i)+C_i\cdot C_i\le 0$, as desired.

Now that we know that $M^{G}$ is finite,  we invoke the signature defect formula \ref{prop:fixtop1}-(3). It tells us that  $0=2\nu-16$, so that $\nu=8$.
Substituting this in the Lefschetz formula \ref{prop:fixtop1}-(1) then gives that $M^{G}$ consists of $8$ points.
Proposition~\ref{prop:ed} shows that  the $\Zb[G]$-module invariants  of  $H_2(M)$ cannot have any cyclotomic summands, and so 
the number of regular and trivial summands will be as asserted.
\end{proof}

Proposition~\ref{proposition:involution} is already enough to prove 
Corollary~\ref{thm:dehntwistproperty}.

\begin{proof}[Proof of Theorem~\ref{thm:dehntwistproperty}]  First note that a Dehn twist $T_S\in\Diff(M)$ induces an element of $O(H_M)$ that is reflection in a $(-2)$-vector.   Suppose that $g:M\to M$ is a diffeomorphism of finite order such that $g_*=(T_S)_*$. If $\la g\ra $ denotes the group generated by $g$, then the existence of the natural surjection $\la g\ra \to \la g_*\ra\cong\Zb/2\Zb$ implies that the order of $g$ is even, say $2m$,  $m\geq 1$.   Then $g^m$ is a smooth involution.   Since $g_*=(T_S)_*$ has order $2$ then either $g^m_*$ equals $(T_S)_*$, so that $g^m\in W_G$, or 
that $(g^m)_*$ is the identity.  Since $g^m$ is a smooth involution which evidently does not satisfy the conclusion of Proposition \ref{proposition:involution}, we get a contradiction.
\end{proof}

Corollary \ref{cor:nikulininvolution} gives a more precise description of how an involution acts on $M$.


\subsection{Resolution of a spin orbifold of dimension 4}\label{subsect:resolution}
Let $G\subset\Diff(M)$ be any finite group.  When $M^G$ has an isolated fixed point $p$ then the quotient $M/G$ is not a manifold, since a neighborhood of the image of $p$ in $M/G$ is the cone on a lens space.  The goal of this section is to explain how one can sometimes  resolve $M/G$ to obtain a new smooth manifold $M'$ equipped with a smooth $G$-action.  We then deduce properties of the original action from properties of $M'$.

We begin with a general discussion of how to construct spin structures on certain resolutions.  We will apply this to the resolved quotient of a K3 surface by a finite group of prime order.

\para{Spin structures} Recall that a spin structure on an oriented  $4$-manifold $W$ amounts to a lift of the structural group of its tangent bundle to the universal cover of $\SL_4(\Rb)$ (which has degree $2$). In other words, if $\SL_W\to W$ denotes the bundle of oriented bases of the tangent spaces of $W$ (this is a principal $\SL_4(\Rb)$-bundle), then a spin structure is given by a double cover $\widetilde{\SL}_W\to \SL_W$, called a \emph{spinor bundle}, which defines a connected (and hence universal) cover of each $\SL_4(\Rb)$-orbit.  A necessary and sufficient condition for its existence  is the vanishing of the second Stiefel-Whitney class $w_2(W)$ of the tangent bundle. It is unique up to isomorphism if and only if the first Stiefel-Whitney class also vanishes.  We note here that if $W$ is closed then the existence of a spin structure on $W$ implies that its intersection form is even: this follows from 
a special case of Wu's formula, which states that  $x\cup x=w_2(W)\cup x$  for all $x\in H^2(W; \Fb_2)$ (see for instance \cite{MS}).

Let $G$ be a finite subgroup of $\Diff(W)$ that preserves the spin structure in the sense  that
every element of $G$ lifts to the spinor bundle  $\widetilde{\SL}_W\to W$. We say that $G$ is  \emph{even} if $G$ admits a lift \emph{as a group} to $\widetilde{\SL}_W$
(so that there  is no need to extend $G$ with the covering transformation of $\widetilde{\SL}_W\to \SL_W$). 

Suppose $W$ is a complex manifold endowed with a nowhere zero holomorphic $2$-form $\omega$. Then 
there is a  reduction of the structural group of the tangent bundle to $\SL_2(\Cb)$ by letting $\SL(\Cb)_W$ be the bundle of complex bases on which $\omega$ takes the value $1$. Since $\SL_2(\Cb)$ is simply connected, the inclusion $\SL_2(\Cb)\subset \SL_4(\Rb)$ lifts canonically to an embedding of $\SL_2(\Cb)$ in the universal cover of $\SL_4(\Rb)$, and hence this gives rise to a spin structure on $W$. So a finite group of diffeomorphisms 
of $W$ that preserves both the complex structure and the form $\omega$ comes with a lift to the associated spinor bundle: it is even.

\para{The resolved orbit space}
In what follows we consider the following special case: $G$ is a nontrivial finite subgroup of  $\SL_2(\Cb)$ and 
$W$ is an open contractible $G$-invariant neighborhood of the origin in $\Cb^2$, for example the unit ball with respect to a $G$-invariant inner product.  Note that this contains the important  case when $G$ is of order two, its nontrivial element being the antipodal involution.
We observe that $G$ then also leaves $(dz_1\wedge dz_2)|_W$ invariant.
The group $G$ acts freely on $W\ssm\{0\}$ and so the $G$-orbit space $W_G$ of $W$ has, as a (normal) singular point, the image of the origin, which we shall denote by $0_G$; topologically it is the open cone over the quotient of a $3$-sphere by a  free $G$-action. 

This singularity has a well known resolution $\pi: W'\to W_G$. Here $W'$ is a complex manifold of dimension $2$, the map $\pi$ is holomorphic and  is an isomorphism over $W_G\ssm 0_G$ and 
$D:=\pi^{-1}(0_G)$ is a normal crossing curve whose irreducible components  are all copies of $\PP^{1}$ having self-intersection $-2$ 
(we call such an irreducible component a \emph{$(-2)$-curve}). They meet each other transversally in such a manner that  the dual intersection graph is the Dynkin diagram of type $A$, $D$ or $E$.  

The case that interests us most is  when $G$ is cyclic of order $n$, in which case $D$ is a chain of $(-2)$-curves, $n-1$ in number, so that the dual intersection diagram is of type $A_{n-1}$. Moreover, the holomorphic form 
$\omega:=(dz_1\wedge dz_2)|_W$, which because of its $G$-invariance descends to a nowhere zero holomorphic $2$-form  on $W_G\ssm 0_G$, extends across $W'$ as a \emph{nowhere-zero holomorphic $2$-form $\omega'$}. This proves that
the $\SL_2(\Cb)$-reduction of the tangent bundle of the $G$-orbit space of $W\ssm \{0\}$ (which is just $W'\ssm D$) extends to $W'$.  
In particular, $W'$ comes with a spin structure.

Conversely, suppose that we are given a pair  $(W', D)$,  where $D$ is a normal crossing divisor that is a deformation retract of the complex surface $W'$ and  
is of the type above: a finite union of  $(-2)$-curves meeting transversally  so that the dual intersection graph is a 
Dynkin diagram of type $A$, $D$ or $E$. 
Then $G=\pi_1(W'\ssm D)$ is finite. The universal cover of $W'\ssm D$ is a complex surface that extends over $W'$ as a complex normal variety that 
has singularities over the singular part $D_\sing$ (the set of normal crossing points) of $D$. These can be resolved in a minimal manner to produce a 
 complex manifold $\tilde W$ and a holomorphic map $\pi':\tilde W\to W'$ such that $\pi'$ is a double cover over $W'\ssm D_\sing$ and the preimage 
$\tilde D$ of $D$ is a normal crossing divisor.   

The normal crossing divisor $\tilde D$ is simply-connected: its irreducible components are copies of $\PP^1$ and the dual intersection graph is a tree. 
In fact, $\tilde D$ contains at least one exceptional curve of the first kind (a \emph{$(-1)$-curve}), and if we contract such curves and keep on doing that 
until none are left to contract, then we have in fact contracted all of $\tilde D$. We thus obtain a proper holomorphic map $\pi: \tilde W\to W$ to a 
complex surface which maps $\tilde D$ to a singleton $\{o\}$ and is an isomorphism over $W\ssm \{o\}$. The naturality of this construction guarantees 
that the $G$-action extends to $\tilde W$
and descends to $W$ with $o$ as its unique fixed point. A nowhere zero holomorphic $2$-form $\omega'$ on $W'$ determines  one on
$W$ which is $G$-invariant and this identifies $G$ with a finite subgroup of $\SL(T_oW)$.

Note that if $n=2$, then $D$ is a single $(-2)$-curve and the construction of $\tilde W$ requires no blowing up:
it is just a double cover of $W'$ ramified over $D$ and the preimage $\tilde D$ of $D$ is a $(-1)$-curve that can be blown down to produce the surface $W$ on which $G$ acts as an antipodal involution.   In a global setting this construction yields the following.

\begin{corollary}[{\bf Resolved orbit space}]\label{cor:resorbitspace}
Let $W$  be an oriented $4$-manifold.  Let $G$ be a finite subgroup of  $\Diff^+(W)$ such for every $p\in W$ the stabilizer $G_p$ acts in some neighborhood of $p$ by an $\SL_2(\Cb)$-model,
in the sense that there is a $\Cb^2$-valued oriented chart in terms of which the stabilizer $G_p$ acts as a subgroup of $\SL_2(\Cb)$. Then the set $F\subset W$ of $p\in W$ for which $G_p\not=\{1\}$ is  discrete.  Further, 
there is a diagram in the category of $G$-manifolds endowed with a $G$-invariant union of submanifolds that meet transversally:
\[
(W,F)\xleftarrow{\pi_W} (\tilde W, \tilde D)\xrightarrow{\pi_{W'}} (W', D)
\]
and it is obtained locally as above: $\pi_W^{-1}F=\tilde D=\pi_{W'}^{-1}D$ with $G$ acting trivially on $W'$ so that
\[
W\ssm F\xleftarrow{\cong} \tilde W\ssm \tilde D\to W'\ssm D
\] 
is the formation of the $G$-orbit space of $W\ssm F$. Both $\tilde D$ and $D$ admit neighborhoods and  complex structures in terms of which
they become normal crossing divisors of the type described above. If $W$ comes with a $G$-invariant spin structure then so does $W'$, and vice versa.
\end{corollary}
\begin{proof}
At every $p\in W$ with $G_p\not=\{1\}$, we choose a $\Cb^2$-valued chart $\kappa: U\hookrightarrow \Cb^2$ that  is orientation preserving and  
$G_p$-equivariant  relative to an embedding $G_p\subset \SL_2(\Cb)$. The given spin structure on this chart  is trivial and the preceding discussion shows that the resulting spin structure on the $G_p$-orbit space of $U\ssm\{p\}$ extends (uniquely) across the resolution of the $G_p$-orbit space of $U$.
\end{proof}

\section{Prime order subgroups of \boldmath$\Diff(M)$}
\label{section:prime:order}

Nikulin determined \cite{nikulin1979} which abelian groups $G$ can appear as automorphism groups of some complex K3 surface $X$ that act trivially on $H^{2,0}(X)$.
In our setting these are precisely the abelian subgroups of $\Diff(M)$ that fix a positive $3$-plane in $\LL_\Rb$ and some complex structure on $M$.
If we assume that $G$ is cyclic of prime order $p$, then the primes $p$ that occur are $2,3,5,7$.  A more detailed study  for the case $p=2$ was given by Morrison \cite{Mor} and Van Geemen-Sarti \cite{vGS} and for the odd primes by Garbagnati-Sarti \cite{GS}. There is in fact a very helpful picture in terms of an elliptic fibration  that enables us to get a good understanding of the topology of such an action. We will say more about this in Section \ref{sect:ellfib}.

In this section $G\subset\Diff(M)$ has prime order $p\geq 2$. \emph{Throughout this section we also assume that $G$ fixes pointwise a 
positive-definite $3$-plane in $\LL_\Rb=H_2(M; \Rb)$.}  This implies that $\Lbf_G$ is the orthogonal complement of $\LL^G$ in $\LL$ and that $\Lbf_G$ is negative definite.

\subsection{A first reduction for the case $p$ odd}

Let $R_G$ be the  set of $(-2)$-vectors in  $\Lbf_G$.  Since $\Lbf_G$ is negative definite,  $R_G$ is  finite. It is in fact a root system whose roots are all of the same length. This implies that its irreducible components are of type $A$, $D$ or $E$.  We show here that there are only two possibilities
which are in a sense opposite extremes.

\begin{theorem}[{\bf Dichotomy}]
\label{thm:dich}
Let $G\subset\Diff(M)$ be a group of prime order $p\geq 3$, and suppose that $G$ fixes pointwise a positive-definite $3$-plane in $H_2(M)$.  Assume that $H_2(M)$ has as a $\Zb[G]$-module no cyclotomic summands. If  $\nu$ is the number of regular summands then $\nu (p+1)\le 24$ with equality holding if and only if  $(p,\nu)\in \{(3,6), (5, 4), (7,3)\}$.  In this case $M^G$ consists of $\nu$ points with the $G$ action on the tangent space at each point factoring through a representation to $\SL_2(\Cb)$.  

Furthermore, precisely one of the following holds:
\begin{description}
\item[(Nikulin type)] $R_G=\emptyset$ and the above inequality is an equality. Moreover there is a possibly different action of $G$  on $M$ 
that acts on $H_2(M)$ as the given one and which preserves  a complex structure and Ricci-flat K\"ahler  metric $M$  and a nonzero holomorphic $2$-form on the resulting complex surface. 
\item[(Coxeter type)]
$R_G$ is a subsystem of type $A_{p-1}^\nu$, spans $\Lbf_G$ and $G$ acts in each $A_{p-1}$-summand as a Coxeter transformation of that summand. 
\end{description}
\end{theorem}

\begin{remarks}\label{rem:}

\begin{enumerate}
\item In the `Nikulin type' case of Theorem~\ref{thm:dich} we do not know whether the given $G$-action preserves a Ricci-flat K\"ahler  structure, but as will become clear below, we 
can come close.  We also suspect that the Coxeter case never occurs.

\item It follows from Corollary 2.2 of Chen-Kwasik \cite{CK} that if $G$  preserves a symplectic form then $M^G$ is finite and  $G$ acts in the tangent space of each fixed point through $\SL_2(\Cb)$. It then follows from Proposition \ref{prop:ed} that $H_2(M)$ has as a $\Zb[G]$-module no cyclotomic summands.  So in this case Theorem~\ref{thm:dich} implies that $\nu (p+1)= 24$, so that the dichotomy there described  holds.
\end{enumerate}
\end{remarks}

We now turn to the proof of Theorem \ref{thm:dich}.

\begin{proof}[Proof of Theorem \ref{thm:dich}] The proof is in three steps. 
\\

\noindent
\emph{Step 1: The action of  $G$ on  $R_G$ is via an element of the Weyl group $W(R_G)$.}

\medskip
Let $v\in R_G$ and let $g\in G\ssm \{1\}$. Then $gv\not=\pm v$ and so $v$ and $gv$ are linearly independent. They therefore span a negative sublattice of rank $2$. Since $v\cdot v=gv\cdot gv=-2$, it follows that $v\cdot gv\in \{-1,0,1\}$. Since $\Lbf_G$ has no nonzero $G$-fixed vectors, we must have 
$\sum_{g\in G} gv=0$. We cannot have $v\cdot gv=0$ for all $g\in G\ssm\{1\}$, for then the $G$-orbit of $v$ consists of pairwise orthogonal vectors and this contradicts  $\sum_{g\in G} gv=0$.
So for some $g\in G$, we will have $v\cdot gv=1$ and it then follows that $Gv$ generates an irreducible root system of rank $\le p-1$.
(This must be of rank exactly $p-1$ with $G$ acting in a standard manner as there are no nontrivial $\Qb[G]$-modules of lower rank.)

Let $R_o$ be an irreducible component of $R_G$. From what we  just proved it follows that $G$ preserves $R_o$. 
We claim that $G$ acts on $R_o$ as an element of the Weyl group $W(R_o)$. If this were not so then it would act nontrivially on the Dynkin diagram of $R_o$.
The Dynkin diagrams of type $A$, $D$ or $E$ with a nontrivial automorphism are $A_r$ ($r\ge 2$), $D_r$ ($r\ge 4$) and $E_6$. For all but $D_4$ this 
is an automorphism group is of order $2$ (and so this cannot happen since we assumed $p$ odd) and for $D_4$ and the group in question is isomorphic to $\Sc_3$. This last case cannot happen either, because then $p=3$ and any order $3$ automorphism of a $D_4$ root system which is not in the Weyl group permutes three mutually perpendicular roots and hence will fix a vector.

\medskip
\noindent
\emph{Step 2: $\nu (p+1)\le 24$ and equality holds if and only if  $(p,\nu)\in \{(3,6), (5, 4), (7,3)\}$ and  $M^G$ consists of $\nu$ points with $G$ acting in the tangent space of each such point through $\SL_2(\Cb)$.}

\medskip
$M^G$ is a union of $m\geq 0$ isolated points $z_1, \dots, z_m$ and $r\geq 0$ connected, embedded subsurfaces $C_1, \dots, C_r$.  Formulas (1) and (3) of Proposition \ref{prop:fixtop1} give 
\begin{gather*}
m+\sum_{j=1}^r \chi (C_j)=24-\nu p,\\
\frac{1}{p-1}\sum_{i=1}^m\Def_{z_i} +\frac{p+1}{3}\sum_{j=1}^r C_j\cdot C_j=\nu p-16, 
\end{gather*}
and  so they combine to give
\begin{equation}\label{eqn:noether0}
 m+ \frac{1}{p-1}\sum_{i=1}^m \Def_{z_i} +\sum_{j=1}^r \big( \chi (C_j)+\frac{p+1}{3}C_j\cdot C_j\big)=8.
\end{equation} 
Since by assumption $H_2(M)$ has as a $\Zb[G]$-module no cyclotomic summands,
Proposition \ref{prop:ed} implies that each $C_j$ has genus zero. In other words, $\chi(C_j)=2$ for each $j$.
As we have seen, each $C_j$ supports a nonzero integral homology class, so that by the adjunction inequality  of Kronheimer-Mrowka (Theorem 1.1 of \cite{KM}),  $C_j\cdot C_j\le -2$. Hence each term of the second sum on the left-hand side of \eqref{eqn:noether0} is $\le 0$. It follows that
\[
m+ \frac{1}{p-1}\sum_{i=1}^m \Def_{z_i} \ge 8.
\]
As we already observed in \S~\ref{subsect:HZsign}, the maximal value of $\Def_{z_i}$ is assumed 
in the case where the action $G$-action on $T_{z_i}M$ factors through $\SL_2(\Cb)$; this value is as Chen-Kwasik note (Lemma 2.4 of \cite{CK})  equal to $\tfrac{1}{3}(p-1)(p-2)$. Thus we get the inequality
\[
 m+ \frac{1}{3}(p-2)m\ge 8,  \text{ or equivalently,  } (p+1)m\ge 24.
\]
On the other hand, $m+\sum_{j=1}^r \chi (C_j)=24-\nu p$ and so $m+\nu p\le 24\le (p+1)m$, which implies that $\nu\le m$ and hence that $\nu (p+1)\le 24$.

Suppose we have equality: $\nu (p+1)= 24$. Then it follows that $\nu=m$ and all the equalities above are equalities: 
$\sum_{j=1}^r (2+\frac{p+1}{3}C_j\cdot C_j)=0$ and since $C_j\cdot C_j\le -2$, this means  that $r=0$ (and so $M^G$ finite) and $(p,\nu)\in \{(3,6), (5, 4), (7,3)\}$
and $\Def_{z_i}=\tfrac{1}{3}(p-1)(p-2)$ for all $i=1, \dots,m$, which means that $G$ acts in $T_{z_i}M$ through $\SL_2(\Cb)$. 
\\

\noindent
\emph{Step 3: Proof of the dichotomy.}

\medskip
Choose a generator $g\in G$ and let  $w\in W(R_G)$ be such  that $w$ acts on $R_G$ as $g_*$. Then $g_*$ and $w$ commute so that either 
$g_*w^{-1}$ has still order $p$ in $\Orth^+(\LL)$ or $g_*=w$. In the first case, $g_*w^{-1}$ generates a subgroup $G_0\subset \Orth^+(\LL)$ of order $p$  with the property that
$L_{G_0}$ contains no $(-2)$-vectors. According to Nikulin \cite{nikulin1979}, this can only happen when $L_{G_0}$ has rank $(p-1)\nu_0 $ with 
$(p, \nu_0)\in\{(3,6), (5,4), (7,3)\}$. In all these cases $(p+1)\nu_0=24$. It is also clear that $\nu_0\le \nu$. But since $(p+1)\nu\le 24$ (by step 1), 
it then follows that $\nu_0=\nu$. So then $\Lbf_G=\Lbf_{G_0}$, meaning that $R_G=\emptyset$ and the action of $G$ being of Nikulin type.

In the second case, we are given that $g$ is contained in a finite reflection group $W$.
We then write $g_*=r_{v_1}\circ \cdots \circ r_{v_m}$ for certain $(-2)$-vectors $v_1,\dots, v_m$ in $R_G$. According to Lemma 3 of \cite{carter1972}, we then may take the $v_i$ to be linearly independent, so that $m=\nu(p-1)$. 

The root system  generated by the $v_1,\dots, v_{\nu(p-1)}$ has only summands of type $A$, $D$ or $E$ and its Weyl group contains $g_*$ as an element of prime order $p$. R.W.~Carter's Table 3 in \cite{carter1972}  tells us that such an element  
is given by a Coxeter transformation of a root subsystem $R_0$, all of whose components  
are of type $A_{p-1}$.  We are given that $\Lbf_G$ is contained in a $\Zb[G]$-submodule of $\LL$ that is isomorphic to $\Zb[G]^\nu$. This means that there exist $v_1, \dots, v_\mu$ in $\LL$
such that $L_G$ has basis 
\[
\{g^j(g-1)v_i\, :\, i=1, \dots , \nu; j=1, \dots, p-1\}.
\]
Since $R_0$ generates a sublattice of $\Lbf_G$ of finite index, this can only happen when
each $(g-1)v_i$ is a $(-2)$-vector. But then it is clear that the above basis must be a root basis of $R_0$. In particular $R_0=R_G$ spans $\Lbf_G$.
\end{proof}

\para{The content in the rest of Section~\ref{section:prime:order}} 
Of the two possibilities in the dichotomy of Theorem \ref{thm:dich}, only the Coxeter case is explicit.
The other, more interesting case,  is when $\Lbf_G$ contains no $(-2)$-vector. Since its existence was established by Nikulin, this  lattice has been the subject of further study (see below for references). Yet such studies 
have mostly been on a case-by-case basis, sometimes aided with computer calculations.

Our take here is somewhat different: it is entirely uniform  (it is not based on a case-by-case treatment and includes the case $p=2$) and leads to a novel characterization of lattices of Nikulin type (see Proposition \ref{prop:Lp}ff). This  will be the key to our proof that the finite prime order  subgroups of $\Diff(M)$ which leave invariant a complex structure  make up a single conjugacy class (Theorem \ref{theorem:classification}). At the same time we analyze the homological aspects of the resolved orbit space construction. Indeed we arrive at the candidate for $\Lbf_G$ via this construction which involves another lattice that we shall refer to as the \emph{Nikulin lattice} (for $p=2$ this was exhibited by Morrison and was thus called by Van Geemen-Sarti). This lattice will also help us to identify the characters of $G$ associated to its representation on the tangent space of a fixed point. We will thus see how close we come to eliminating the Coxeter type in Theorem~\ref{thm:dich}.  We shall see  for $p=2$ that we will always have a Nikulin-type situation and indeed, we will find that then the answer is: quite close.

\subsection{The resolved orbifold $M_G$} 
We assume in this subsection that if $p$ is odd then $G$ acts at each fixed point according to an $\SL_2(\Cb)$-model (when $p=2$ then $M^G$ is finite by Proposition~\ref{proposition:involution},  so this is then always the case).
We denote by $\nu$ the number of rational cyclotomic summands in $\LL_{ \Qb}$. This is then also the number of fixed points of $G$ in $M$ 
and we have $(p+1)\nu=24$ (either by Proposition \ref{proposition:involution} or Theorem \ref{thm:dich}).
We begin by applying  the construction of  Corollary \ref{cor:resorbitspace}.

\begin{proposition}\label{prop:class}
Let $p\geq 2$.  When $p>2$ assume that the group $G$  acts at each fixed point according to an $\SL_2(\Cb)$-model.
Then $\Lbf_G^\vee/\Lbf_G$ is an $\Fb_p$-vector space of dimension $\nu$.  Further, if we form the resolved orbit space as in Corollary \ref{cor:resorbitspace}:
\[
(M, M^G)\xleftarrow{\pi_{M}} (\tilde M, \tilde D)\xrightarrow{\pi_{M'}} (M',D),
\]
then $M'$ is homeomorphic to a $K3$ manifold and the rank of $H_2(M'\ssm D)$ is $2\nu-2$.
\end{proposition}
\begin{proof}
It is clear that  $M'$ is a  oriented closed $4$-manifold. It is simply-connected since the inclusion $M'\ssm D\subset M'$ induces a surjection on fundamental groups; $\pi_1(M'\ssm D$) is cyclic of order $p$ and a generator will die in $M'$ near a point in $D$ of total ramification.

Since the irreducible components of $D$ have an intersection matrix that is negative definite, the map $H_2(D; \Qb)\to H_2(M'; \Qb)$ is injective
and has image a negative-definite subspace of dimension $\nu (p-1)$ (there are $\nu$ pairwise disjoint copies of chains of length $p-1$.)  Its orthogonal complement is the image of  $H_2(M'\ssm D'; \Qb)$.
The  transfer map identifies 
$H_2(M'\ssm D; \Qb)$ with $H_2(M\ssm M^G; \Qb)^G$ and multiplies the intersection pairing with $p$. The latter has signature $(3, 19-(p-1)\nu)$. Hence 
$H_2(M')$ has signature $(3, 19-(p-1)\nu+ (p-1)\nu)=(3,19)$. Since $M$ has a unique spin structure, $G$ must leave that structure invariant  and so  $M'$ inherits 
a spin structure. Wu's formula then implies that the intersection form on $M'$ is even.  
 Freedman's classification of closed, simply-connected $4$-manifolds (Theorem 1.5 of \cite{Fr})
 then implies that $M'$ must be homeomorphic to a 
$K3$ manifold.

The rank of $H_2(M'\ssm D)$ is $22$ minus the rank of $H_2(D)$. Since the latter is $\nu (p-1)$, we get $22- \nu p+\nu=(-2+\nu)+\nu=2\nu-2$.
\end{proof}

We continue with the situation of Proposition \ref{prop:class}.  Consider part of the exact sequence of the Gysin sequence for the pair $(M',D)$:
\[
H^1(M'\ssm D)\to H_2(D)\to H^2(M')\to H^2(M'\ssm D)\to H_1(D).
\]
The first and the last term are zero: since $\pi_1(M'\ssm D)=G$ it follows that $H^1(M'\ssm D)=\Hom (G, \Zb)=0$ and the simple-connectivity of $D$ implies that $H_1(D)=0$. The result is a short exact sequence with the middle term free abelian. 

Let $\Nbf_p$ denote the primitive hull of $H_2(D)$ in $H^2(M)$ so that $\Nbf_p/H_2(D)$ can be identified with the torsion of $H^2(M'\ssm D)$.
By  the universal coefficient theorem,  the torsion subgroup of $H^2(M'\ssm D)$
is $\ext^1(G,\Zb)$ and hence cyclic of order $p$. It follows that  $\Nbf_p$ contains $H_2(D)$ as a sublattice of index $p$. The lattice $H_2(D)$ is isomorphic to $\Abf_{p-1}^{\oplus \nu}$. Since the discriminant group $\Abf_{p-1}^\vee/\Abf_{p-1}$ is cyclic of order $p$, 
the lattice $H_2(D)$ (as embedded in $H^2(M')$) has discriminant group an $\Fb_p$-vector space of dimension $\nu$. Since $\Nbf_p$ contains  $H_2(D)$ as a sublattice of index $p$, the discriminant group $\Nbf_p^\vee/\Nbf_p$ is an $\Fb_p$-vector space of dimension $\nu-2$.
Essentially by construction we now have an exact sequence
\[
0\to \Nbf_p\to H^2(M')\to H^2(M'\ssm D)_{\tf}\to 0,
\] 
where tf stands for the torsion free quotient. The duality between $H^2(M'\ssm D)_{\tf}$ and $H_2(M'\ssm D)$  is realized inside
$H^2(M'; \Qb)$: if we regard $H_2(M'\ssm D)$ as embedded in $H^2(M')$ via $H_2(M'\ssm D)\to H_2(M')\cong H^2(M')$, then it is simply
the orthogonal complement of $\Nbf_p$ in $H^2(M')$ and $H^2(M'\ssm D)_{\tf}$ is
identified with  $H_2(M'\ssm D)^\vee\subset H_2(M'\ssm D; \Qb)\subset H^2(M';\Qb)$ as the image of $H^2(M')$ under the orthogonal projection along $\Nbf_p$.  
In particular, this identifies the discriminant groups of $H_2(M'\ssm D)$ and  $\Nbf_p$ and under this
identification the discriminant quadratic forms are each others opposite. Since the group $\Nbf_p^\vee/\Nbf_p$ is a $\Fb_p$-vector space, $pH^2(M'\ssm D)\subset H_2(M')$ and hence the intersection pairing on $H^2(M'\ssm D)_\tf$ takes values in $\frac{1}{p}\Zb$. In other words, if we multiply the pairing on 
$H_2(M')$ with $p$, we obtain an integral  form  on $H^2(M'\ssm D)_\tf$: it makes $H^2(M'\ssm D)_\tf(p)$ a lattice.

\begin{proposition}\label{prop:Nsequence}
The  map $\pi_{M'}^*$ identifies $H^2(M')/\Nbf_p\cong H^2(M'\ssm D)_\tf$ with $H^2(M)^G$, but multiplies the intersection pairing with $p$.  It also defines an injective  morphism of exact sequences
of which the top sequence is naturally split:
\begin{center}
\begin{tikzcd}
0\arrow{r}&  H_2(\tilde D) \arrow{r}{\textrm{cycle}} &H^2(\tilde M)^{G} \arrow{r} &H^2(M)^{G} \arrow{r}&0\\
0\arrow{r}&  \Nbf_p\arrow{r}{\textrm{cycle}}\arrow{u}&H^2(M') \arrow{r}\arrow{u}{\pi_{M'}^*} &H^2(M'\ssm D)_\tf\arrow{r}\arrow{u}{\cong}& 0\\
\end{tikzcd}
\end{center} 
Dually, the map $\pi_{M'*}$ induces an isomorphism $H_2(M)_G\cong H^2(M'\ssm D)$ that multiplies the intersection pairing with $p$.  It defines a morphism of exact sequences of which the top sequence is naturally split: 
\begin{center}
\begin{tikzcd}
0&  H^2(\tilde D)\arrow{l}\arrow{d} &H_2(\tilde M)_{G}\arrow{l}\arrow{d}{\pi_{M'*}}  &H_2(M)_{G} \arrow{l}\arrow{d}{\cong}&0 \arrow{l}\\
0&  \Nbf_p^\vee\arrow{l}&H_2(M') \arrow{l}&H_2(M'\ssm D)\arrow{l}&  \arrow{l}0\\
\end{tikzcd}
\end{center} 

\end{proposition}
\begin{proof}
We first examine the unramified $G$-cover $M\ssm M^G \to M'\ssm D$. For such a cover we have a spectral sequence
\[
E^{p,q}_2=H^p(G; H^q(M\ssm M^G))\Rightarrow H^{p+q}(M'\ssm D).
\]
Since $M^G$ is finite and $M$ is a simply connected 4-manifold, $H^1(M\ssm M^G)=0$ and $H^2(M)\to H^2(M\ssm M^G)$ is an isomorphism. 
So $E^{p,1}_2=0$. On the other hand, $E^{p,0}_2=H^p(G; \Zb)$ is trivial in odd degrees and a copy of $\Fb_p$ in positive even degrees, whereas $E^{0,q}_2=H^q(M\ssm  M^G)^{G}$. It follows that $E^{p,q}_\infty=E^{p,q}_2$ when $p+q=2$ so that we  
have a short exact sequence
\begin{equation}
0\to \Fb_p\to H^2(M'\ssm D)\to H^2(M)^{G}\to 0.
\end{equation}
This establishes the isomorphism $H^2(M'\ssm D)_\tf\cong H^2(M)^G$. 
The isomorphism $H_2(M)_G\cong H^2(M'\ssm D)$ is obtained in a similar manner by means of the  spectral sequence
\[
E_{p,q}^2=H_p(G; H_q(M\ssm M^G))\Rightarrow H_{p+q}(M'\ssm D),
\]
where we now need to observe that again $E_{p,1}^2=0$, but that $E_{p,0}^2=H_p(G; \Zb)$ is a copy of $\Fb_p$ in odd positive degrees and trivial in even positive degrees.

The lower exact sequence of the first diagram maps to the exact sequence of the pair 
$(\tilde M, \tilde M\ssm \tilde D)$, of which the relevant part is  
\[
H^1(M\ssm M^G)\to H_2(\tilde D)\to H^2(\tilde M) \to H^2(M\ssm M^G) \to H_1(\tilde D).
\]
The extremal terms vanish:  we already noted that $H^1(M\ssm M^G)=0$ and $H^2(M\ssm M^G)\cong H^2(M)$ and since the connected components of $\tilde D$ are simply connected,
$H_1(\tilde D)=0$. This not only gives a short exact sequence and also identifies it (via Poincar\'e duality on $M$ and $\tilde M$) with the one associated to the pair $(\tilde M, \tilde D)$:
\[
0\to H_2(\tilde D)\to H_2(\tilde M)\xrightarrow{\pi_{M*}} H_2(M)\to 0,
\]
and this sequence is split by $\pi_M^*$. It therefore remains exact if we pass to $G$-invariants. This proves the exactness of the top sequence.

Since $\pi_{M'_*}$ takes the fundamental class of $\tilde M$ to $p$ times the fundamental class of $M'$, the vertical maps 
multiply the intersection forms with $p$. 
The splitting property follows from the fact that $\tilde M\to M$ is an iterated blow-up: the image of the cycle map $H_2(\tilde D)\to H^2(\tilde M)$ supplements the image of $H^2(M)\to H^2(\tilde M)$
and equals its orthogonal complement with respect to the intersection pairing.

The assertions regarding the second diagram are proved  in a similar fashion. 
\end{proof}

As we will make clear in Subsection \ref{subsect:nikulin} below, Proposition \ref{prop:Nsequence} implies that the Nikulin lattice completely determines the discriminant form of $\Lbf_G$ and hence its  `genus'.

\subsection{The Nikulin lattice}\label{subsect:nikulin}

In this subsection we give a complete description of $\Nbf_p$ as a lattice.  Number the fixed points of $G$ on $M$ by $p_1, \dots , p_\nu$
and write $D_i$ resp.\ $\tilde D_i$ for the connected component of $D$ resp.\  $\tilde D$ over $p_i$. We number the irreducible components
of $D_i$ by $D_{i,1},\dots , D_{i,p-1}$ so that $D_{i,j}$ and $D_{i,j'}$ are disjoint unless $|j-j'|\le 1$. These irreducible components come with an 
orientation and we will identify them with the homology classes in $H_2(D)$ that they represent. Thus $H_2(D_i)$ is as a lattice isomorphic with
$\Abf_{p-1}(-1)$ and the basis $\{ D_{i,j}\}$ of $H_2(D)$ consists of $(-2)$-vectors and makes up a system of simple roots.

The group $H_2(D_i)^\vee/H_2(D_i)$ is cyclic of order $p$. The way we numbered  the basis of $H_2(D_i)$ yields a natural representative for a generator, namely $\varpi_i:=\frac{1}{p}\sum_{j=1}^{p-1} jD_{i,j}$. This is in fact  the antipode of a fundamental weight: $\varpi_i\cdot D_{i,j}=0$ unless $j=p-1$, in which case we get $-1$. This also shows that $\varpi_i\cdot\varpi_i=\frac{1-p}{p}$. 
Note however that if we number the elements of $D_i$ in the opposite direction (replacing $D_{i,j}$ by $D_{i,p-j}$), then this replaces $j$ by $p-j$. 
 
A generator of $\Nbf_p/H_2(D)$ can be represented by some $\varpi\in H_2(D)^\vee$. We write this as 
$\varpi:=\sum_{i=1}^\nu k_i\varpi_i$ with $k_i\in \{0,1,\dots,p-1\}$.
If we replace in the exact sequence above $M'$ by a regular neighborhood $U_i$ of $D_i$, then the preimage of $U_i$ will still be connected and so
the same argument shows that $k_i\not=0$. But since the numbering of the elements of $D_i$ in the opposite direction replaces $D_{i,j}$ by $D_{i,p-j}$ (and hence $k_i$ by $p-k_i$), only the image of $k_i^2$ in $\Fb_p^\times$ is well-defined. The following proposition not only determines these images, but also gives a concrete geometric interpretation of them.

\begin{proposition}\label{prop:Nikulinform}
After possible renumbering, the image of $(k_1^2, \dots, k_\nu^2)$ in $(\Fb_p^\times)^\nu$ is for $p=2,3,5,7$ equal to  respectively $(1,1,1,1,1,1,1,1)$, 
$(1,1,1,1,1,1)$, $(1,1,4,4)$ and $(1,4,2)$. If we then let $(k_1,\dots, k_\nu)$ equal respectively  $(1,1,1,1,1,1,1,1)$, 
$(1,1,1,1,1,1)$, $(1,1,2,2)$ and $(1,2,3)$ and let
\[
\varpi:=\sum_{i=1}^\nu \frac{k_i}{p}\sum_{j=1}^{p-1} jD_{i,j},
\]
then $\Nbf_p=H_2(D)+\Zb \varpi$ is an even overlattice of $H_2(D)$ whose $(-2)$-vectors are those of $H_2(D)$. Furthermore, if
$g\in G$ has characters $\chi_i, \chi_i^{-1}$ in $T_{p_i}M$ (for a complex structure on $T_{p_i}M$ compatible with the given orientation), then
$\{\chi_i, \chi_i^{-1}\}=\{\chi_1^{k_i}, \chi_i^{-k_i}\}$.

This completely determines the isomorphism types of the lattices $H_2(M'\ssm D)$ and $\LL^G$.
\end{proposition}

\begin{proof}
Since $\varpi\in \Nbf_p$, we must have 
\[
\frac{1-p}{p}\sum_{i=1}^\nu k_i^2=\varpi\cdot \varpi\in 2\Zb, \text{ or equivalently, } \frac{1-p}{2}\sum_{i=1}^\nu k_i^2\equiv0\pmod{p}.
\]
If $p=2$ then there is nothing to verify, except perhaps to observe that $\frac{-1}{2}\cdot 8$ is indeed even.  
If $p$ is odd then $\frac{1-p}{2}\in \Fb_p^\times$ and so this is then equivalent to $
\sum_{i=1} k_i^2\equiv0 \pmod{p}$. The squares in $\Fb_3^\times$, $\Fb_5^\times$, $\Fb_7^\times$ are resp.\    $\{1\}$, $\{1,4\}$, $\{1,2,4\}$.
The linear combinations that are divisible by $p$ admit a unique solution: up to a permutation the image of $(k_1^2,\dots , k^2_\nu)$ in $(\Fb_p^\times)^\nu$ is equal to resp.\ $(1,1,1,1,1,1)$, $(1,1,4,4)$ and $(1,2,4)$. This clearly fixes the isomorphism type of $\Nbf_p$. For the given choices of $(k_1, \dots, k_\nu)$ we see
that $\varpi\cdot \varpi$ is resp. $-4$, $-4$, $-8$, $-12$, and one can check that this does not introduce any new $(-2)$-vectors in $\Nbf_p$.

For the assertion regarding the action of $G$ on the tangent spaces at the fixed points, let  $W_i$ be the ball-like neighborhood of $p_i$ that we used to complex-linearize the $G$-action at $p_i$. So we obtain a regular neighborhood $W'_i$ of $D_i$ in $M'$ for which $W'_i\ssm D_i\subset M'\ssm D$ induces an isomorphism on fundamental groups and hence  defines an isomorphism $G\cong H_1(W'_i\ssm D_i)$. The exact sequences for the pair $(W'_i, W'_i\ssm D_i)$ and $(M, M\ssm D)$ produce the commutative diagram
\[
\begin{tikzcd}
H_2(M)\arrow{r}\arrow{d} &H^2(D)\arrow{r}\arrow{d} &H_1(M\ssm D) \arrow{r}\arrow{d}{\cong}&0\\
H_2(D_i)\arrow{r}&H^2(D_i) \arrow{r}&H_1(W'_i\ssm D_i)\arrow{r}& 0\\
\end{tikzcd}
\]
in which the middle map is given by projection. The upper sequence identifies $H_1(M\ssm D)$ with the cokernel of $H_2(M')\to H^2(D)$,  which is just 
$H_2(D)^\vee/ \Nbf_p^\vee$. The element $\varpi\in\Nbf_p$ defines a map  $x\in H_2(D)^\vee\mapsto x\cdot \varpi\in\Qb$
takes values in $\frac{1}{p}\Zb$ and is of course $\Zb$-valued on $\Nbf_p^\vee$. Indeed, its reduction modulo $\Zb$ defines an isomorphism 
$H_2(D)^\vee/ \Nbf_p^\vee\cong \frac{1}{p}\Zb/\Zb$, so that $x\mapsto \exp(2\pi\sqrt{-1} x\cdot\varpi)$ defines a nontrivial character $\chi$ on $G$.
On the other hand, the lower sequence identifies $H_1(W'_i\ssm D_i)$ with the cokernel of 
$H_2(D_i)\to H^2(D_i)\cong H_2(D_i)^\vee$ which has the image of $\varpi_i$ as a  generator. Via the identification $H_1(W'_i\ssm D_i)\cong G$ 
the character $\chi$ takes on this generator the value $\exp(2\pi\sqrt{-1}\varpi_i\cdot\varpi)=\exp(2\pi\sqrt{-1}\varpi_i\cdot k_i\varpi_i)=\exp(2\pi\sqrt{-1}k_i/p)$, and so  the assertion follows.

The discriminant group of $H_2(M'\ssm D)$ is canonically isomorphic to the one of $\Nbf_p$ (and so is an $\Fb_p$-vector space of dimension $\nu-2$) and under this isomorphism their  discriminant quadratic forms are  opposite. By Proposition \ref{prop:Nsequence}, $H_2(M'\ssm D)$ has rank $2\nu-2$ and its intersection pairing is indefinite. Since that rank is at least two more than the number ($\nu-2$) of generators of the discriminant group of 
$H_2(M'\ssm D)$,  this completely determines the isomorphism type of $H_2(M'\ssm D)$ by Corollary 1.13.3 of  \cite{nikulin1980}.  This is then also true for  $\LL^G$ since 
\[\LL^G\cong H^2(M'\ssm D)_\tf(p)=H_2(M'\ssm D)^\vee(p)\]
by Proposition~\ref{prop:Nsequence}.
\end{proof}

\begin{remark}\label{rem:ellfibrationcont}
The numbers $k_i$ in Proposition~\ref{prop:Nikulinform} are recognized in terms of an elliptic fibration  (see Figure \ref{figure:fiberaction} and Subsection \ref {sect:ellfib} for more details) as the rotation numbers: we can label the  $I_p$-fibers 
$F_1,\dots, F_\nu$ such that $g\in G$ rotates in $F_i$ over $k_i$ times the angle over which it rotates in $F_1$.
\end{remark}

\begin{lemma}\label{lemma:nikulinlattice}
The lattice $\Nbf_p$ has the following properties.
\begin{enumerate}
\item The $(-2)$-vectors of $\Nbf_p$ are those of $H_2(D)$; that is, the elements of the root system $R=\sqcup_{i=1}^\nu  R_i$, with each $R_i$ of type $A_{p-1}$ with root basis 
$\{D_{i,j}\}_{j}$.
\item The discriminant group of $\Nbf_p$ is represented by the set of $\Fb_p$-linear combinations $\sum_{i=1}^\nu l_i\varpi_i$ for which $\sum_i k_il_i=0$ modulo the span of $\varpi=\sum_i k_i\varpi_i$.
\item  The group of orthogonal transformations of $\Nbf_p$ that act trivially on its discriminant group $\Nbf_p^\vee/\Nbf_p$ is the Weyl group $W(R)=\prod_i W(R_i)$ of $R$.
\end{enumerate}
\end{lemma}
\begin{proof}
For  $(k_1, \dots, k_\nu)$ as above, we see
that $\varpi\cdot \varpi=(1-p)/p.\sum_i k_i^2$ is resp. $-4$, $-4$, $-8$, $-12$ and one may check that this does not introduce any new $(-2)$-vectors in $\Nbf_p$.
This shows that $R$ is as claimed. As to the other statement, we shall assume that $p$ is odd, leaving the case $p=2$ as an exercise (which amounts to a minor modification of the proof below).

The orthogonal group of $\Nbf_p$ is a semi-direct product of the Weyl group $W(R)$ and a 
permutation group of the collection $\{D_{i,j}\}_{i,j}$. This permutation group has as a subgroup the permutations which preserve each irreducible component.
That subgroup can then also be represented by the transformations which induce in each component plus or minus the identity (minus the identity is not in 
a Weyl group of type $A_{p-1}$ if $p>2$). Hence we can realize this group also as one which permutes the $\varpi_i$'s up to sign. Consider the inclusions
\[
H_2(D)\subset \Nbf_p\subset \Nbf_p^\vee\subset H_2(D)^\vee . 
\]
The discriminant group $H_2(D)^\vee/H^2(D)$ is an $\Fb_p$-vector space with basis $\{\varpi_i\}_{i=1}^\nu$ and $\varpi=\sum_i k_i\varpi_i$ defines an isotropic vector in  it that defines a generator of $ \Nbf_p/H_2(D)$. Since $\varpi_i\cdot \varpi_i=(1-p)/p$, it follows that
$\Nbf_p^\vee/H_2(D)$ is represented in terms of this basis by the  hyperplane of  $l=(l_1, \dots l_\nu)\in \Fb_p^\nu$ with $\sum_i l_ik_i=0$.  So
$\Nbf_p^\vee/\Nbf_p$ is identified with the quotient of this hyperplane by the span of $k=(k_1, \dots, k_\nu)$. The group of permutations of the basis vectors of $ \Fb_p^\nu$ up to sign acts faithfully on this quotient and thus the lemma follows.
\end{proof}

We now describe our candidate for $\Lbf_G$ in terms of the Nikulin lattice $\Nbf_p$. For this, let $\varrho\in H_2(D)^\vee$ be the {\em Weyl vector}, characterized by the property that it takes the value $1$ on each $D_{i,j}$.  So if we write this out on the basis $\{D_{i,j}\}_{i,j}$,  then 
\[
\varrho=\sum_{i=1}^\nu \varrho_i \text{  with   }  \varrho_i=-\tfrac{1}{2}\sum_{j=1}^{p-1} j\cdot(p-j) D_{i,j}.
\]
When $p$ is odd, this is in fact an element of $H_2(D)$. We compute
\begin{equation}\label{eqn:rhosquared}
\varrho\cdot \varrho=\sum_{i=1}^\nu \varrho_i\cdot\varrho_i=-\frac{\nu}{2}\sum_{j=1}^{p-1} j\cdot (p-j)=-\frac{\nu(p-1)p(p+1)}{2\cdot 6}=-2(p-1)p,
\end{equation}
where we used that $\nu(p+1)=24$. In particular, $\varrho\cdot \varrho\equiv0\pmod{p}$.

We  note that $\varrho_i\cdot \varpi_i$ is minus the coefficient of  $D_{i,p-1}$ in $\varrho_i$, which is $-\frac{1}{2}(p-1)$, and so
\begin{equation}\label{eqn:rhovalue}
\textstyle \varrho \cdot \varpi=\sum_{i=1}^\nu \varrho_i\cdot k_i\varpi_i=\frac{1}{2}(p-1)\sum_{i=1}^\nu  k_i.
\end{equation}

\subsection{A Leech type sublattice of the Nikulin lattice}

The following lattice will play a central role in the rest of this paper.

\begin{definition}[{\bf The lattice \boldmath$\Lbf_p$}]
\label{def:}
Let  $\Lbf_p\subset\Nbf_p$ to be  the sublattice of the vectors $x\in \Nbf_p$ on which $\varrho\cdot D_{i,j}\equiv 0\pmod{p}$ for all $i,j$.
\end{definition}

Note that by the above computation  $\varrho$ lies in $\Lbf_p$.
 We will see that $\Lbf_p$ is a lattice which possesses some remarkable properties. To illustrate this, let us observe that for $p=2$ this reproduces the lattice $\Ebf_8(-2)$. Indeed, 
$\Lbf_2(-\frac{1}{2})$ is by definition the set of $(x_1, \dots, x_8)\in(\frac{1}{2}\Zb)^8$ for which the each difference $x_i-x_j$ is integral and $\sum_i x_i$ is even and this is  one of the classical descriptions of  $\Ebf_8$ \cite{serre1970}.

The  Weyl group $W(R_i)$ contains an element $\sigma_i$ of order $p$ that takes $D_{i,j}$ to $D_{i,j+1}$ for $j=1, \dots, p-2$, and $D_{i,p-1}$ to $D_{i,0}:=-\sum_{j=1}^{p-1}D_{i,j}$, and $D_{i,0}$ back to $D_{i,1}$ (it is a Coxeter element). Let $\sigma=(\sigma_1^{k_1}, \dots, \sigma_\nu^{k_\nu})\in W(R)$. Any element of $W(R)$, so in particular $\sigma$, leaves invariant any overlattice of $H_2(D)$ such as $\Nbf_p$.

\begin{proposition}\label{prop:Lp}
The sublattice  $\Lbf_p$ of $\Nbf_p$ is of index $p$, and contains $\varrho$. It has also the following properties
\begin{enumerate}
\item $\Lbf_p$ is negative definite of rank $\nu(p-1)$ and contains no $(-2)$-vector(\footnote{In the terminology of Nikulin, $\Lbf_p$ is a lattice  of Leech-type.}). 
\item The discriminant group $\Lbf_p^\vee/\Lbf_p$ is an  $\Fb_p$-vector space of dimension $\nu$.
\item The group $C_p$ of automorphisms of $\Lbf_p$ that act  trivially on $\Lbf_p^\vee/\Lbf_p$ is cyclic of order $p$ (and generated by $\sigma$).
\end{enumerate}
\end{proposition}
\begin{proof} The proof is in three steps. 
\\

\noindent
\emph{Step 1: The sublattice  $\Lbf_p$ of $\Nbf_p$ is of index $p$, contains $\varrho$, but no $(-2)$-vector and  $\Lbf_p^\vee/\Lbf_p$ is an  $\Fb_p$-vector space of dimension $\nu$.}

\medskip
It is clear that $x\in \Nbf_p\to \varrho\cdot x\in \Zb$ is onto and so $\Lbf_p$ is a sublattice of $\Nbf_p$ of index $p$ as claimed. 
Hence $|\Lbf_p^\vee/\Lbf_p|=p^2\cdot |\Nbf_p^\vee/\Nbf_p|=p^\nu$. The assertion that $\Lbf_p^\vee/\Lbf_p$ is an $\Fb_p$-vector space then amounts to the property that $p\Lbf_p^\vee\subset \Lbf_p$. Since
$\frac{1}{p}\varrho$ maps to a generator of $\Lbf_p^\vee/\Nbf_p^\vee$, it follows that 
$p\Lbf_p^\vee=p\Nbf_p^\vee+\Zb \varrho$. The latter is a subgroup of $\Nbf_p$ on which $\varrho$ takes values in $\Zb p$, hence is contained in $\Lbf_p$.  

We now check that $\Lbf_p$ contains no $(-2)$-vectors. Since we already know that all the $(-2)$-vectors in $\Nbf_p$ lie in $H_2(D)$, 
it suffices to show that $\varrho$ takes on every root in $H_2(D)$ a value that is not divisible by $p$. To this end, we  note that every positive root in $H_2(D_i)$ is of the form $\sum_{j=r}^{s} D_{i,j}$ with $1\le r<s\le p-1$; on such a root  $\varrho$ takes the value $r-s$ and this is never divisible by $p$.
\\

\noindent
\emph{Step 2: $\sigma$ preserves  $\Lbf_p$ and acts trivially on  $\Lbf_p^\vee/\Lbf_p$.}

\medskip 
We must show that $\sigma -1$ takes $\Lbf_p^\vee$ to $\Lbf_p$. We again use that  $\Lbf_p^\vee=\Nbf_p^\vee+\frac{1}{p}\Zb \varrho$ and that $\Nbf_p^\vee/H_2(D)$ is represented by the $\Zb$-linear combinations $\sum_i l_i\varpi_i$ for which $\sum_i k_il_i\equiv 0\pmod{p}$. We therefore first check this on $H_2(D)$ (regarded as a sublattice of $\Nbf_p^\vee$). 

For this we note that  $\sigma (D_{i,j})-D_{i,j}$ (with $j=1, \dots, p-1$) is of the form
$D_{i, j'}-D_{i,j}$ where $j'$ can also take the value $0$. Its  inner product with $\varrho$ is therefore zero or $-p$, so that this is an element of $\Lbf_p$.

Next we check this for $\sum_i l_i\varpi_i$ with  $\sum_i k_il_i\equiv 0\pmod{p}$. We compute that 
\[
\sigma_i^{k_i}(\varpi_i)=\varpi_i-\sum_{j=k_i}^{p-1}D_{i,j}
\]
and so
\[
\sigma(\sum_i l_i\varpi_i)-\sum_i l_i\varpi_i= -\sum_{i=1}^\nu l_i\sum_{j=k_i}^{p-1}D_{i,j}.
\]
To test whether the right hand side lies in $\Lbf_p$, we take its inner product with $\varrho$. This gives $-\sum_i l_i(p-k_i)=\sum_i k_il_i-p\sum_i l_i$ and this is indeed a multiple of $p$. This proves that $\sigma-1$ takes $\Nbf_p^\vee$ to $\Lbf_p$.

We finally show that $(\sigma-1)(\frac{1}{p}\varrho)\in \Lbf_p$.
A straightforward  computation shows that
\[
\textstyle (\sigma -1)(\frac{1}{p}\varrho_i)= (\sigma_i^{k_i} -1)(\frac{1}{p}\varrho_i)=k_i\varpi_i-\sum_{j=1}^{k_i} jD_{i,p-j}
\] 
so that  
\[
\textstyle (\sigma-1)(\frac{1}{p}\varrho)= \varpi -\sum_{i=1}^\nu \sum_{j=1}^{k_i}jD_{i,p-j}.
\]
Using Equation \ref{eqn:rhovalue} we see that the value of $\varrho$ on the right-hand side is 
\begin{equation}
\label{eq:sumdiv1}
\sum_{i=1}^\nu \tfrac{1}{2}k_i(p-1)-\sum_{i=1}^\nu  \tfrac{1}{2} k_i(k_i-1)=-\tfrac{1}{2}\sum_{i=1}^\nu k_i(k_i-p).
\end{equation}
If $p=2$ then the right-hand side of \eqref{eq:sumdiv1} is $-\tfrac{1}{2}\cdot 8=-4$, which is indeed even; 
if $p$ is odd then the fact that $\sum_i k_i^2$ is divisible by $p$ implies that the right-hand side of \eqref{eq:sumdiv1} is also divisible by $p$. So in all cases, $(\sigma-1)(\frac{1}{p}\varrho)\in \Lbf_p$.
\\

\noindent
\emph{Step 3: Any automorphism of $\Lbf_p$ that acts  trivially on $\Lbf_p^\vee/\Lbf_p$ is a power of $\sigma$.}

\medskip

First note that the $W(R_i)$-orbit of $\varpi_i$  is the $p$-element set  $\{\varpi_i-\sum_{j=p-r}^{p-1}D_{i,j}\}_{r=0}^{p-1}$ and
that this realizes $W(R_i)$ as the full permutation group of that orbit. By taking the inner product with $\varrho_i$ and then reducing modulo $p$, this orbit gets identified with $\Fb_p$. We thus faithfully represent $W(R)$ as a permutation representation of  $\Fb_p^\nu$.

Let $g\in \Orth(\Lbf_p)$ act trivially on $\Lbf_p^\vee/\Lbf_p$. Since $\Nbf_p$ is an overlattice of $\Lbf_p$, it then follows that $g$ preserves $\Nbf_p$ and acts trivially on $\Nbf_p^\vee/\Nbf_p$. According to Lemma \ref{lemma:nikulinlattice}, $g$ will then be an element of the Weyl group $W(R)$.
This element must preserve $H_2(D_i)\cap \Lbf_p$. Now an element of $R_i$ (the set $(-2)$-vectors  in $H_2(D_i)$) is up to sign of the form
$\sum_{j=r}^s D_{i,j}$ for some $1\le r<s\le p-1$ and on that vector $\varrho$ takes the value $r-s$. This shows that  the elements of $R_i$ on which $\varrho$ takes a value $1\pmod{p}$ is  $\{D_{i,j}\}_{j=0}^{p-1}$. It is clear that $g$ must permute these roots. Notice that the 
``Dynkin diagram'' of $\{D_{i,j}\}_{j=0}^{p-1}$ (given by their inner products) is that of a $p$-gon.  Hence $g$ acts on this $p$-gon as a rotation: it cannot reverse its orientation because such a transformation will not be in the Weyl group $W(R_i)$ and so it will be a rotation. This means that $g$ acts in $H_2(D_i)$ as some power $\sigma^{n_i}_i$ of the Coxeter transformation $\sigma_i$.
 
Let $l=(l_1, \dots ,l_\nu)\in \Zb^\nu$ be such that $\sum_i k_il_i\equiv 0\pmod{p}$, so that $\varpi(l):=\sum_i l_i\varpi_i\in \Lbf_p^\vee$. We are given that $(g-1)\varpi(l)\in \Lbf_p$. In  particular, $\rho\cdot \big((g-1)\varpi(l)\big)\equiv 0 \pmod{p}$. 
Now $g(\varpi_i)-\varpi_i=\sigma_i^{n_i}(\varpi_i)-\varpi_i= -\sum_{j=n_i}^{p-1}D_{i,j}$ and so
\[
0\equiv \rho\cdot \big((g-1)\varpi(l))=\sum_{i=1}^\nu l_i\rho_i\cdot(\sum_{j=n_i}^{p-1}-D_{i,j})=\sum_{i=1}^\nu -l_i(p-n_i)\equiv\sum_{i=1}^\nu l_in_i\pmod{p}.
\]
As this must be true for all $l=(l_1, \dots ,l_\nu)\in \Zb^\nu$ with $\sum_il_i\equiv 0\pmod{p}$, it follows that 
$(n_1, \dots, n_\nu)$  is proportional to $(k_1, \dots , k_\nu)$.  This proves that $g=\sigma^{n_1}$.
\end{proof}

The lattices $\Lbf_p$ appear in the literature implicitly beginning with the work of Nikulin. Although he does not exhibit the lattices themselves, he  proves a uniqueness 
assertion which we  state as follows.

\begin{proposition}[Nikulin \cite{nikulin1979}, Thm.\ 4.7]\label{prop:unique}
Let $\Lbf$ be a negative-definite lattice without $(-2)$-vectors which admits an orthogonal automorphism of prime order $p$ that acts trivially on 
$\Lbf^\vee/\Lbf$. If $\Lbf$ admits an embedding in the K3 lattice $\LL$ then $p\in \{2,3,5,7\}$ and $\Lbf\cong \Lbf_p$. 
\end{proposition}

An alternative proof of Theorem~\ref{prop:unique} using Niemeier lattices was later given by Hashimoto \cite{Ha}. The lattices themselves have been exhibited by
Morrison \cite{Mor} and Van Geemen-Sarti \cite{vGS} for $p=2$ and by Garbagnati-Sarti \cite{GS} for the odd primes in a case-by-case manner (they  denote the lattice $\Lbf_p$ by  $\Omega_p$).

The   $\Lbf_p$ comes with the structure of a rank $\nu$ module over $\Zb[C_p]$.  Since the co-augmentation acts trivially, this factors through a  $\Zb(C_p)$-module structure.   The $C_p$-invariance of the symmetric bilinear form on $\Lbf_p$ implies that the  map 
$h:\Lbf_p\times\Lbf_p\to  \Zb[C_p]$ defined by 
\[h(u,u'):=\sum_{g\in C_p} (u\cdot gu' ) g\] 
is a hermitian form, by which we mean that $h$ is 
$\Zb[C_p]$-linear in the first variable and that $\overline{h(u,u')}=h(u',u)$, where the bar is the anti-involution in $\Zb[C_p]$ induced by  the inversion $g\in C_p\mapsto g^{-1}\in C_p$. If we compose $h$ with the augmentation we get 
$\sum_{g\in C_p} (u\cdot gu' )=0$ and so $h$ takes its values in the augmentation ideal. 
For that reason, we may as well  compose  $h$ with the projection $\Zb[C_p]\to \Zb(C_p)$. We thus recover for odd $p$ (up to scalar)  the hermitian modules found by Garbagnati-Sarti \cite{GS} (they also identify $\Lbf_3$ as the Coxeter-Todd lattice). {\it A priori} the  group $\Orth(\Lbf_p)$  will for odd $p$ only induce semi-linear transformations for this structure (meaning that they can induce a nontrivial Galois automorphism of $\Zb(C_p)$);  these will all be linear precisely when $C_p$ is central in $\Orth(\Lbf_p)$.

Theorem 4.1 of \cite{GS} states that 
\begin{description}
\item[$p=3$]  implies $\LL^G\cong \Ubf\oplus\Ubf(3)^{\oplus 2}\oplus\Abf_2(-1)^{\oplus 2}$,
\item[$p=5$] implies $\LL^G\cong \Ubf\oplus\Ubf(5)^{\oplus 2}$, 
\item[$p=7$] implies $\LL^G\cong \Ubf(7)\oplus (\begin{smallmatrix} 2 & 1\\ 1 & 4\end{smallmatrix})$ (the second summand can be identified with the orthogonal complement of an embedding of $\Abf_6$ in $\Ebf_8$).
\end{description}

\begin{remark}\label{rem:nopush}
Recall that the $\nu$-dimensional $\Fb_p$-vector space $\Lbf_p^\vee/\Lbf_p$ comes with a nondegenerate quadratic form.   There is an evident homomorphism $\Orth(\Lbf_p)\to \Orth(\Lbf_p^\vee/\Lbf_p)$, which by Proposition  \ref{prop:Lp} has as kernel the group generated by $\sigma$.  We will denote this group by $C_p$.  
For $p=2$ this map is onto: we already noticed that $\Lbf_2\cong \Ebf_8(-2)$ so that
$\Lbf_2^\vee/\Lbf_2\cong \Ebf_8\otimes\Fb_2$, and it is well-known that the orthogonal group of the latter is the Weyl group of $\Ebf_8$ modulo  its center (see for example Bourbaki, \emph{Groupes et Alg\`ebres de Lie}, 
Ch.\ 6, Exercice \S 4, 1). This also implies that the discriminant form on $\Lbf_2$ is different than that of $\Abf_1(-2)^{\oplus 8}$. 
On the other hand, one can show that for $p$ odd,  $\Lbf_p$ has the same genus as $\Abf_{p-1}^{\oplus \nu}$. This is one reason why we are not able to push the classification as far as for $p=2$. 
\end{remark}

\subsection{The case $p=2$}

Using what we have proved so far in this section, we can now prove Theorem~\ref{theorem:nikulin:type}.

The following corollary illustrates the power of Proposition \ref{prop:Nikulinform}.

\begin{corollary}\label{cor:nikulininvolution}
In case $p=2$, we have $\Lbf_G\cong \Ebf_8(-2)$. In particular, $\Lbf_G$ contains no $(-2)$-vectors. In fact, $G$ then defines an involution of Nikulin type in the sense of Definition \ref{def:nikulintype}: there exists an isomorphism 
$\LL\cong \Ubf^{\oplus 3}\oplus\Ebf_8(-1)$ via which $G$ acts by exchanging the $\Ebf_8(-1)$-summands.
\end{corollary}

\begin{proof}
We found that the Nikulin lattice $\Nbf_2$ has a discriminant group an $\Fb_2$-vector space of dimension $6$. It is not hard to check that its discriminant form is that of 
$\Ubf(2)^{\oplus 3}$. It follows from Nikulin's uniqueness theorem that in the situation of Propositions \ref{prop:Nsequence} and  \ref{prop:Nikulinform},  the orthogonal complement $H_2(M'\ssm D)$ of $\Nbf_2$ in $H_2(M')$ is isomorphic to
$\Ubf(2)^{\oplus 3}\oplus \Ebf_8(-1)$. The dual of this lattice is $\Ubf^{\oplus 3}(\frac{1}{2})\oplus \Ebf_8(-1)$  and hence  $\LL^G\cong \Ubf^{\oplus 3}\oplus \Ebf_8(-2)$ by Proposition \ref{prop:Nsequence}. This is turn implies that $\Lbf_G$ is a lattice whose discriminant form is that of $\Ebf_8(-2)$. 

Since $\Lbf_G$ has rank $8$ and is negative definite, it follows that $\Lbf_G(\frac{1}{2})$ is unimodular negative definite. It cannot be odd, because then $\Lbf_G\cong \Abf_1(-1)^{\oplus 8}$ and this has not the discriminant form of $\Ebf_8(-2)$. It follows that $\Lbf_G\cong \Ebf_8(-2)$ as asserted.

It then also follows that the embedding $\LL^G\oplus \Lbf_G\subset \LL$ is isomorphic to $(\Ubf^{\oplus 3}\oplus \Ebf_8(-2))\oplus \Ebf_8(-2)\subset
\Ubf^{\oplus 3}\oplus\Ebf_8(-1)$, where the two copies of $\Ebf_8(-2)$ embed in $\Ebf_8(-1)^{\oplus 2}$ as resp.\ the diagonal and the antidiagonal.
It is clear that via such an isomorphism, $G$  simply exchanges the $\Ebf_8(-1)$-summands.
\end{proof}

Since we arrived at this Corollary \ref{cor:nikulininvolution} via a symplectic resolution, we obtain additional geometric information, for 
this also tells us how to recover the $G$-action topologically: start off with an oriented  smooth manifold  $M'$ that is 
homeomorphic to a  K3 surface and find in it a disjoint union $D$ 
of $8$ embedded $2$-spheres with self-intersection $-2$ such that $H^2(D; \Zb)$ meets $H^2(M')$ in  $H^2(M'; \Qb)$ in $\Nbf_2$. This implies that
$M'$ admits a double covering $\tilde M'\to  M'$ ramified along $D$. The preimage $\tilde D$ of $D$ in $\tilde M'$ consists $8$ pairwise disjoint $2$-spheres with self-intersection $-1$. These can be contracted to produce a smooth $4$-manifold homeomorphic to the K3 surface. The covering 
transformation then gives us the $G$-action. 

\begin{remark}\label{rem:}
This comes close to, but does not quite produce, a classification of $\Diff(M)$-conjugacy classes of involutions of Nikulin type.
In the complex-analytic setting the resolved orbit space $M'$ produces a genuine K3 surface and one can show that the involutions thus obtained 
lie in a single conjugacy class in $\Diff(M)$. It is not clear to us whether every involution of the type considered here belongs to that class.
For example, we do not know whether our $M'$ is always \emph{diffeomorphic} (rather than homeomorphic) with a K3 surface.
But even if this were the case, then we would still need to know whether the isotopy classes 
of `Nikulin submanifolds' $D\subset M'$ that define a given primitive embedding of the Nikulin lattice in $H_2(M')$  make up a single orbit under the Torelli group of $M'$.
\end{remark}

\begin{proof}[Proof of Theorem~\ref{theorem:main2}] \label{section:main2:proof} 
Let $G\subset\Mod(M)$ be a subgroup of order $2$.  Ruberman (Theorem 2.1 of \cite{Ru}) and Matumoto \cite{Ma2} proved that if a closed spin $4$-manifold $X$ with $H_1(X)=0$ admits a smooth involution acting trivially on $H_2(X)$, then $X$ has signature $0$.  Since the K3 manifold $M$ has $H_1(M)=0$ and is spin, but has signature $-16$, it follows that the induced representation $G\to \Aut(H_M)\cong \Orth^+(\LL)$ is faithful. We can therefore identify $G$ with its image in $\Orth^+(\LL)$.

The space $\Gr^+_3(\LL_\Rb)$ of positive-definite 3-planes in $\LL_\Rb$ 
is the symmetric space of $\Orth^+(\LL_\Rb)$, whose natural $\Orth^+(\LL_\Rb)$-metric has nonpositive sectional curvature.  Thus any finite $G$ acting on this space has a fixed point; that is, $G$  leaves invariant 
some positive-definite $3$-plane $P\subset \LL_\Rb$.  There are two cases to consider according to whether 
or not the induced $G$-action on $P$ is trivial.

If the $G$ action on $P$ is nontrivial then  $\Lbf_G=0$.  Theorem~\ref{thm:K3examples} then implies that $G$ can be realized  as a group of automorphisms of a complex structure on $M$ endowed with a Ricci-flat K\"ahler metric.  In particular the subgroup $G\subset\Mod(M)$ can be realized as a group of diffeomorphisms.

If the $G$-action on $P$ is trivial then Corollary \ref{cor:nikulininvolution} implies that  a necessary condition for $G$ to lift to $\Diff(M)$ is that
$\Lbf_G$ contain no $(-2)$-vector. But Theorem~\ref{thm:K3examples} shows that this is also sufficient:  then too  
$G$ can then be realized as a group of automorphisms of a complex structure on $M$ endowed with a Ricci-flat K\"ahler metric.
\end{proof}

\subsection{Nielsen realization relative to a Ricci-flat metric}
In this subsection we deduce some geometric consequences of our study of $\Lbf_p$.

\begin{lemma}[{\bf Moduli space of Ricci-flat metrics}]
\label{lemma:ricciflatmoduli}
Let  $i:\Lbf_p\hookrightarrow \LL$ be a primitive embedding. Extend the $C_p$ action to $\LL$ by letting it act trivially on $i(\Lbf_p)^\perp$.  Denote by $\Rc(i)$ the space of 
Ricci-flat metrics $s$ on  $M$ for which $C_p$ lies in the image of the isometry group of $(M,s)$. Then the orbit space of $\Rc(i)$ with respect to the natural action of the Torelli group  of $M$ can  be identified with the complement of a locally-finite union of submanifolds of codimension $3$ in the symmetric space of 
positive $3$-planes in the orthogonal complement of $i$ in $\LL_\Rb$; in particular, it is a connected manifold of dimension $3(2\nu-5)$.
This manifold is the base of a fine moduli space: it supports a universal family of Ricci-flat K3 manifolds.

Conversely, if $s$ is a Ricci-flat metric on $M$, then every cyclic group of isometries of $(M, s)$ of order $p$ 
that pointwise fixes a positive $3$-plane in $\LL_\Rb$ 
arises in this manner for some primitive embedding $i:\Lbf_p\hookrightarrow \LL$.
\end{lemma}
\begin{proof}
Since $C_p$ acts trivially on $\Lbf_p^\vee/\Lbf_p$, its action on $\Lbf_p$ extends to $\LL$ by letting it act as the identity on the orthogonal complement of
$i$; this is the extension to which our notation $i_*C_p$ refers.  In this way $i_*C_p$ is a subgroup of $\Orth(\Lbf_p)$.

If $s$ is a Ricci-flat metric on $M$, then in view of the Torelli theorem, the isometry group of $(M,s)$ acts faithfully on $\LL$ and a  necessary and sufficient condition that this group contains $C_p$ is that the positive $3$-plane $P_s\subset \LL_\Rb$ defined by its self-dual harmonic $2$-forms has the property that
its orthogonal complement $P_s^\perp$ (so that is the space of antiselfdual harmonic $2$-forms) does not contain a $(-2)$-vector and contains the image of $i$. In other words, the moduli space in question is the space of positive $3$-planes $P\subset i(\LL)_\Rb^\perp$ for which $P^\perp$ does not contain a $(-2)$-vector. The space of positive $3$-planes $P\subset i(\LL)_\Rb^\perp$ is the symmetric space $D_i$ of the orthogonal group of $i(\LL)_\Rb^\perp$, and the locus of positive $3$-planes $P\subset i(\LL)_\Rb^\perp$ that are also perpendicular to a given $(-2)$-vector $v$ defines a symmetric subspace of codimension $3$. These symmetric subspaces are locally finite in $D_i$ and hence the moduli space is as asserted. It is known that this support a universal family (see Corollary 5.5 of \cite{Lo}).

If $s$ is a Ricci-flat metric on $M$ and $G$ is a subgroup of isometries of  $(M, s)$ of order $p$ which pointwise fixes a positive $3$-plane in $\LL_\Rb$, then
$P_s^\perp$ contains $\Lbf_G$. In particular, $\Lbf_G$ does not contain any $(-2)$-vector and is therefore by Nikulin's uniqueness property \ref{prop:unique} of the above type.
\end{proof}

Building on all of the above material, we can now finish the proof of Theorem~\ref{theorem:classification}, which classifies all prime order groups $G\subset\Diff(M)$ preserving some complex structure.

\begin{proof}[Proof of Theorem \ref{theorem:classification}]
Let us first observe that in the situation of Lemma \ref{lemma:ricciflatmoduli} the embedding $i$ determines an orbit of order $p$ subgroups of $\Diff(M)$ with respect to the action of the Torelli group. This is because we have associated to $i$ a fine moduli space with connected base, whose monodromy representation on $\LL$ is trivial.  

We now let $i$ vary. The orthogonal complement for every primitive embedding of $\Lbf_p$ in $\LL$ is even and indefinite of rank $2\nu-2$ (it has signature $(3, 2\nu-5)$) and its discriminant group of $\Lbf_p$ has $\nu$ generators. Since $\nu<2\nu-2$, it follows from Nikulin's Thm.\ 1.14.1
and Prop.\ 1.14.2 \cite{nikulin1980} that such primitive embeddings  satisfy the weak form of Witt's theorem, by which is meant that 
the sublattices $\Lbf\subset \LL$ isomorphic to $\Lbf_p$ make up a single orbit under $\Orth(\LL)$. Since $\Orth(\LL)=\{\pm 1\}\Orth^+(\LL)$, it then follows that it is also a single orbit under $\Orth^+(\LL)$. If we combine this with the preceding, we find that the cyclic subgroups $G$ of 
$\Diff(M)$ of order $p$  which preserve a Ricci-flat  metric and for which $L_G$ is negative definite make up a single conjugacy class. 

For the proof of Theorem \ref{theorem:classification} it remains to make the passage from invariant Ricci-flat metric metrics to invariant complex structures. As we recalled earlier \cite{HKLR}, the twistor theory shows that a Ricci-flat metric  $s$ on $M$ is a K\"ahler metric for a $2$-sphere of complex structures, and that every complex structure so arises. In fact, for every complex structure $J$ on $M$ and every automorphism $g$ of $(M,J)$ for which  $\Lbf_{\la g\ra}$ is negative definite, there exists a Ricci-flat metric $s$ 
that is a K\"ahler metric relative to $J$ and for which $g$ is an isometry.
\end{proof}

\subsection{Elliptic fibrations}\label{sect:ellfib}
As we saw,  if a cyclic group $G$ of prime order acts faithfully on a K3 surface, then $p\in \{2,3,5,7\}$.  Van Geemen-Sarti (for $p=2$) and  Garbagnati-Sarti (for $p$ odd) show that this can then be realized as a subgroup of the  Mordell-Weil group of an elliptic  K3 surface $X$ with $\nu=24/(p+1)$ 
fibers of Kodaira type  $I_p$ and as many fibers of  type $I_1$.  (We recall here that a $I_m$-fiber is for $m>1$ a `cycle' of $(-2)$-curves, $m$ in number and $I_1$ is a rational curve with a node.) So the sum of the Euler characteristics of the singular fibers  is $\nu(p+1)=24$, which is indeed the  Euler characteristic of $M$.  The $G$-action is then realized as the {\em Mordell-Weil group} - the group of automorphisms that preserve each fiber and act on it as a translation. 
Each nontrivial element of $G$ acts in each $I_p$-fiber  as a nontrivial rotation (but not necessary over the same angle, see below). It also acts nontrivially in each $I_1$-fiber and as observed earlier, the singular points of such  fibers  account for all the $G$-fixed points (see Figure~\ref{figure:fiberaction}).  

In what follows we give a homological description of the elliptic fibration (with and without a section) on which $C_p$ acts. The geometrically most useful one is the case when $C_p$ preserves each fiber, but as we will see for $p\not=7$, there exist also fibrations that permutes the fibers.

\para{Elliptic fibrations with fiberwise $C_p$-action} We here describe the homological model of an elliptic fibration that is fiberwise preserved by the $C_p$-action, as depicted in Figure \ref{figure:fiberaction}.

Let $\Zbf$ be a copy of $\Zb$, but regarded as a rank one lattice with zero quadratic form.  Denote
its generator by $f$.  Define an embedding $i: \Nbf_p\to \Nbf_p\oplus\tfrac{1}{p}\Zbf$ by $x\in \Nbf_p\mapsto x+\tfrac{1}{p}(\varrho\cdot x)f$ and let
\[\hat \Nbf_p:=i(\Nbf_p)+\Zbf \subset  \Nbf_p\oplus\tfrac{1}{p}\Zbf.\] The definition of $\Lbf_p$ shows that $\hat \Nbf_p$
meets the first summand in $\Lbf_p$ and the second summand in $\Zbf$.    For $i=1, \dots , \nu $ and $j=1,\dots , p-1$ let 
\[\tilde D_{i,j}:=i(D_{i,j})=D_{i,j} +\tfrac{1}{p}f\in \hat\Nbf_p\] so that 
$\tilde D_{i,0}:=D_{i,0}+\frac{1}{p}f\in \hat\Nbf_p$.  Also recall that $D_{i,0}=-\sum_{j=1}^{p-1} D_{i,j}$ and $\sum_{j=0}^{p-1} \tilde D_{i,j}=f$.  Extend the $C_p$-action to $\Nbf_p\oplus \tfrac{1}{p}\Zbf$  by letting it act trivially on $\tfrac{1}{p}\Zbf$. Since $\varrho$ is an element of $p\Lbf_p^\vee$, 
and $\sigma-1$ takes $\Lbf_p^\vee$ (and hence $\Nbf_p$) to $\Lbf_p$, it follows that $i(\sigma-1)$ takes its values in $\Zbf$. In other words,
 $\hat \Nbf_p$ is $C_p$-invariant.

Suppose we are given primitive embedding $\hat \Nbf_p\hookrightarrow \LL$. 
If $M$ is endowed with a complex structure $J$ for which the image of  this embedding is the Picard lattice (i.e., the orthogonal 
complement of $H^{2,0}(M, J)$), then the theory of K3 surfaces ensures that after possibly multiplying this embedding with $-1$ 
(in order to make it compatible with the spinor orientation), $f$ is the class of an elliptic fibration.   This elliptic fibration has for each $i=1,\dots,\nu$ a singular fiber of 
type $I_p$ whose irreducible components have as classes the images of $\{\tilde D_{i,j}\}_{j=0}^{p-1}$,  and $\nu$ other singular fibers of type $I_1$. 
The Torelli theorem ensures that $C_p$ uniquely lifts to the automorphism group of $(M,J)$ as 
in Figure (\ref{figure:fiberaction}) so that $\Lbf_{C_p}$ is then just the image of $\Lbf_p$.

\begin{remark}\label{rem:}
If the singularities of the quotient $X/G$ are resolved in a minimal manner (as in \S~\ref{subsect:resolution}) we obtain a K3 surface $X'$ that is also 
elliptically fibered, but one for which the $G$-orbit space of each $I_p$-fiber resp. $I_1$-fibers of $X$ yields an 
$I_1$-fiber resp.\  resolves to an $I_p$-fiber of $X'$. So $X'$ is of the same type as $X$. In each smooth fiber 
$F$ the group $G$ is contained in its translation group $T(F)$ (an elliptic curve for which $F$ is a principal homogeneous space); it is in fact a subgroup of order $p$ of the $p$-torsion group $T(F)[p]$. The $G$-orbit space $F'$ 
of $F$  has translation group $T(F)/G$ and the image of $T(F)[p]$ in $T(F)/G$ is a group that is via the Weil pairing (or equivalently, the symplectic form on $H_1(F)$) naturally identified with $G':=\Hom(G, \Cb^\times)$. 
One can show that $G'$ in fact acts on $X'$ and that repeating this construction with $(X', G')$ returns a copy of $(X,G)$.
\end{remark}

\para{Elliptic fibrations with $C_p$-action permuting the fibers} 
The advantage of the above model is that it helps us  to understand the geometry of the $C_p$-action. But a K3  surface 
$X$ with an elliptic fibration and an automorphism of order $p$ is not necessarily of the above type, 
as we will see below this can happen when $p=2,3,5$ (the algebraic counterpart of this observation is somewhat implicit in \S 5 of 
Garbagnati-Sarti \cite{GS}), but not  for $p=7$. 
We here use that  for such $p$ there exists an embedding $i:\Lbf_p\hookrightarrow \LL$ such that $i(\Lbf_p)^\perp$ contains a copy  of $\Ubf$
(as is shown in \cite{GS}). Let $f$ be one of the isotropic generators of this copy so that $i(\Lbf_p)\oplus \Zb f$ is a primitive sublattice of $\LL$.  
The Torelli theorem guarantees the existence of a complex structure $J$ on $M$ such that $X:=(M,J)$ has $i(\Lbf_p)\oplus \Zb f$ as its Picard 
lattice. This  K3 surface comes with a faithful action of   $C_p$ and $\pm f$ is the class of an elliptic fibration $X\to B\cong\Pb^1$ that will be left 
invariant under the $C_p$-action. We want to understand how $C_p$ acts on this fibration.

Since $X$ has no $(-2)$-curves, the Kodaira classification of singular fibers implies that these must be all irreducible: either nodal or cuspidal. So if a fiber $X_t$ is singular and $C_p$-invariant, then the number of fixed points of $C_p$  it contains is equal to its Euler characteristic: one when $X_t$ is nodal (the fixed point is its singular point) and two when $X_t$ is cuspidal (the singular point and a smooth point). This implies  that $C_p$ cannot preserve each fiber, for then the number of fixed points would at least be equal to the Euler characteristic of $X$, which is  $24$ and this is  $>\nu$. So $C_p$ acts nontrivially on the base $B\cong \Pb^1$ of the pencil.
In terms of an affine coordinate $t$ on $B$ this action will be as multiplication with a primitive $p$th root of unity, so that the fibers 
$X_0$ (over $t=0$) and  $X_\infty$ (over $t=\infty$),  contain the fixed points.  Each of these comes with an automorphism of order $p$ and the number of the fixed points 
must add up to $\nu$.  If $X_0$ (or $X_\infty$) is smooth, then for $p=2$ it can have four fixed points (minus the identity with respect to a fixed point)  and for $p=3$ it can have  three (but only when $X_0$ has $j$-invariant zero). Singular fibers different from $0, \infty$ lie over free $C_p$-orbits in $\PP^1\ssm\{0, \infty\}$. This then leads to the following possibilities only.

For $(p,\nu)=(2,8)$, $X_0$ and $X_\infty$ must be \emph{smooth} and each must contain four fixed points,  
the involution being the standard one described above. So there will be in  general $12$ free $C_2$-orbits of singular fibers.  So we will have in general 24 nodal curves. The 24 points in $\Pb^1\ssm\{0, \infty\}$ will consist of $12$ antipodal pairs.

For $(p,\nu)=(3,6)$, $X_0$ and $X_\infty$ must be smooth and  each  contains
three fixed points (so they have $j$-invariant zero) and there will be in general eight $C_3$-orbits of singular fibers.

For $(p,\nu)=(5,4)$,  $X_0$ and $X_\infty$ must each be cuspidal and contain two 
fixed points and there will be in general four $C_5$-orbits of singular fibers.

This type of reasoning also shows that the case $(p,\nu)=(7,3)$ cannot occur. 
For then the three fixed points are contained in the union of the invariant fibers $X_0$ and $X_\infty$, so that one them (say $X_0$) must be 
cuspidal (the other is then nodal). 
At the cusp, we then choose complex-analytic coordinates $(z_1,z_2)$ so that  
$t=z_1^3+z_2^2$ and the $C_p$-action is linearized. It is then clear that $C_p$ must act on the tangent space with characters 
$(\chi^2, \chi^3)$ with $\chi$ primitive. This however contradicts the fact that it acts through $\SL_2(\Cb)$.  
The impossibility of this case  answers a question raised by Garbagnati-Sarti at the end of their paper \cite{GS}.

\para{The lattice of an elliptic fibration with a section} We only discuss the more interesting  case,  when $C_p$ preserves the fibers.
Elliptic fibrations endowed with a section are obtained by enlarging $\hat \Nbf_p$ to a lattice $\Kbf_p=\hat \Nbf_p+\Zb s$ whose quadratic form  is determined by stipulating  that $s$ is perpendicular to $i(\Nbf_p)$, $s\cdot f=1$ and $s\cdot s=-2$. 
It is then clear that the orthogonal complement of $i(\Nbf_p)$ in $\Kbf_p$ is spanned by $f$ and $s$. The basis $\{(e:=s+f,f)\}$ identifies this
orthogonal complement with $\Ubf$, so that \[\Kbf_p=i(\Nbf_p)\oplus\Ubf\cong \Nbf_p\oplus \Ubf.\]

Let us also determine the orthogonal complement of $\Lbf_p$ in $\Kbf_p$. This clearly contains $f$. Since  for $x\in \Nbf_p$,
\[
0=s\cdot i(x)=s\cdot (x+\tfrac{1}{p}(\varrho\cdot x)f)=s\cdot x+\tfrac{1}{p}\varrho\cdot x=(s+\tfrac{1}{p}\varrho)\cdot x,
\]
we see that $ps+\varrho\perp \Lbf_p$. This also shows that $s\cdot\varrho=-\tfrac{1}{p}\varrho\cdot\varrho=2(p-1)$ and so
\[(ps+\varrho)\cdot (ps+\varrho)=-2p^2+4p(p-1)-2(p-1)p=-2p.\] Since $s\cdot f=1$, it follows that 
$e':=ps+\varrho +f$ is isotropic and that $\Lbf_p^\perp=\Zb e'+\Zb f$ is a copy of $\Ubf(p)$.
It also follows that the  $C_p$-action on $\Lbf_p$ extends to $\Kbf_p$ by letting it fix $e'$ and $f$, for  then 
$(\sigma -1)(s)=(\sigma -1)(-\varrho/p)$ and we know that $\sigma -1$ takes $\varrho/p\in \Lbf_p^\vee$  to $\Lbf_p$.

The algebro-geometric interpretation is that for a primitive  embedding of $\Kbf_p\hookrightarrow \LL$ that is compatible with the spinor orientation there exist complex  structures $J$ on $M$ for which the image of  $\Kbf_p$ is the Picard lattice, $f$ the class of an elliptic fibration, and $s$ the class of a section. The group $C_p$ then lifts uniquely to $(M,J)$ so that every $C_p$-translate of this section is also one and the image 
of $\Lbf_p$ is $\Lbf_{C_p}$.

\bigskip{\noindent
Dept. of Mathematics, University of Chicago\\
E-mail: bensonfarb@gmail.com\\
\\
Yau Mathematical Sciences Center, Tsinghua U.\ (Beijing),
U.\ of Chicago, U.\ Utrecht\\
E-mail: e.j.n.looijenga@uu.nl

\end{document}